\documentclass{amsart}
\usepackage{amsmath,amssymb, amsbsy}
\usepackage{amsthm}
\usepackage{color,psfrag}
\usepackage[dvips]{graphicx}
\usepackage{enumerate}
\usepackage[]{hyperref}

 \usepackage{mathrsfs}

\usepackage{color}
\usepackage{amsmath, amsthm, amssymb}
\usepackage{amsfonts}
\usepackage[ansinew]{inputenc}
\usepackage{graphicx}
\usepackage{caption}
\usepackage{subcaption}

\usepackage[english]{babel}
\usepackage{hyperref}

\usepackage{tikz}
\usepackage{rotating}

\usepackage{cite}
\usepackage{amscd}
\usepackage{color}
\usepackage{bm}
\usepackage{enumerate}

\usepackage{verbatim}
\usepackage{hyperref}
\usepackage{amstext}
\usepackage{latexsym}
\theoremstyle{plain}
\newtheorem{theorem}{Theorem}[section]

\newtheorem{remark}[theorem]{Remark}

\numberwithin{theorem}{section}
\numberwithin{equation}{section}

\newcommand{\average}{{\mathchoice {\kern1ex\vcenter{\hrule height.4pt
width 6pt depth0pt} \kern-9.7pt} {\kern1ex\vcenter{\hrule
height.4pt width 4.3pt depth0pt} \kern-7pt} {} {} }}

\def\R{\mathbb{R}}

\renewcommand{\a }{\alpha }
\renewcommand{\b }{\beta }
\renewcommand{\d}{\delta }

\newcommand{\D }{\Delta }

\newcommand{\e }{\varepsilon }
\newcommand{\g }{\gamma}

\newcommand{\G }{\Gamma}

\newcommand{\n }{\nabla }
\newcommand{\vp }{\varphi }
\renewcommand{\phi}{\varphi}

\renewcommand{\t }{\tau }

\renewcommand{\th }{\theta }

\renewcommand{\O }{\Omega }

\newcommand{\ov}{\overline}

\newcommand{\be}{\begin{equation}}
\newcommand{\ee}{\end{equation}}
\newcommand{\de}{\partial}

\newcommand{\ti}{\widetilde}

\newcommand{\ra}{{\rangle}}
\newcommand{\la}{{\langle}}

\newcommand{\N}{\mathbb{N}}

\newcommand{\C}{\mathbb{C}}

\newcommand{\cB}{{\mathcal B}}
\newcommand{\cC}{{\mathcal C}}

\newcommand{\cP}{{\mathcal P}}

\newcommand{\cR}{{\mathcal R}}
\newcommand{\cS}{{\mathcal S}}

\newcommand{\cU}{{\mathcal U}}

\renewcommand{\epsilon}{\varepsilon}

\newcommand{\weakly}{\rightharpoonup}

\renewcommand{\geq }{\geqslant}
\renewcommand{\le }{\leqslant}
\renewcommand{\leq }{\leqslant}

\newtheorem{Theorem}{Theorem}[section]
\newtheorem{Corollary}[Theorem]{Corollary}
\newtheorem{Lemma}[Theorem]{Lemma}
\newtheorem{Proposition}[Theorem]{Proposition}
\theoremstyle{definition}

\begin{document}

\title[Asymptotic expansions at Dirichlet-Neumann boundary  junctions]{Asymptotic expansions and unique
 continuation\\ at Dirichlet-Neumann boundary  junctions\\ for planar elliptic equations}

\author[Mouhamed Moustapha Fall]{Mouhamed Moustapha Fall}
\address{\hbox{\parbox{5.7in}{\medskip\noindent{M.M. Fall\\
          African Institute for Mathematical Sciences (A.I.M.S.) of Senegal,\\
         KM 2, Route de Joal, AIMS-Senegal\\
         B.P. 1418.
         Mbour, Senegal. \\[3pt]
         {\em{E-mail address: }}{\tt mouhamed.m.fall@aims-senegal.org.}}}}}
\author[Veronica Felli]{Veronica Felli}
\address{\hbox{\parbox{5.7in}{\medskip\noindent{V. Felli\\
Universit\`a di Milano--Bicocca,\\
        Dipartimento di Scienza dei Materiali, \\
        Via Cozzi
        55, 20125 Milano, Italy. \\[3pt]
        \em{E-mail address: }{\tt veronica.felli@unimib.it.}}}}}
\author[Alberto Ferrero]{Alberto Ferrero}
\address{\hbox{\parbox{5.7in}{\medskip\noindent{A. Ferrero\\
Universit\`a degli Studi del Piemonte Orientale,\\
Dipartimento di Scienze e Innovazione Tecnologica, \\
Viale Teresa Michel 11, 15121 Alessandria, Italy. \\[3pt]
        \em{E-mail address: }{\tt alberto.ferrero@uniupo.it}.}}}}

\author[Alassane Niang ]{Alassane Niang}
\address{\hbox{\parbox{5.7in}{\medskip\noindent{A. Niang\\       
         Cheikh Anta Diop University of Dakar (UCAD), \\
         Department of Mathematics, Faculty of Science and Technics, \\ 
         B.P. 5005. Dakar, Senegal. \\[3pt]
         {\em{E-mail address: }}{\tt  alassane4.niang@ucad.edu.sn.}}}}}

\thanks{2010 {\it Mathematics Subject Classification.} 35C20,  35J15, 35J25.\\
  \indent {\it Keywords.} Mixed boundary conditions, unique
  continuation, asymptotic expansion, monotonicity formula.\\
}

\date{\today}

 \begin{abstract}
   \noindent
We consider    elliptic equations in planar domains with mixed boundary conditions
   of Dirichlet-Neumann type. Sharp asymptotic expansions of  the solutions   and
   unique continuation properties from the Dirichlet-Neumann junction
are proved.
 \end{abstract}

\maketitle
\section{Introduction}\label{sec:intro}

The present paper deals with elliptic equations in planar
domains with mixed boundary conditions and aims at proving
asymptotic expansions and unique
continuation properties for solutions near
boundary points where a transition from Dirichlet to Neumann boundary
conditions occurs.

A great attention has been devoted to the problem of unique continuation
for solutions to partial differential equations starting from the paper
by Carleman
\cite{carleman}, whose  approach was based on some weighted
a priori inequalities. An alternative approach to unique continuation
was developed by Garofalo and Lin \cite{GL} for elliptic equations in divergence
form with variable coefficients, via
local doubling properties and Almgren monotonicity formula. The
latter approach has the advantage of giving not only unique continuation but also
precise asymptotics of solutions near a fixed point, via a suitable combination of
monotonicity methods with blow-up analysis, as done in \cite{FF, FF2, FFT,
  FFT2, FFT3}. The method based on doubling properties and
Almgren monotonicity formula has also been  successfully applied to treat
the problem of
unique continuation from the boundary in \cite{ae,aek,FF,tz2} under
homogeneous Dirichlet conditions and in \cite{tz1}
under homogeneous Neumann conditions. Furthermore, in
\cite{FF} a sharp asymptotic description of the behaviour of
solutions at conical boundary points was given  through a fine blow-up
analysis.  In  the present paper, we   extend the procedure
developed in  \cite{FFT,
  FFT2, FFT3,FF} to the case of mixed
Dirichlet/Neumann
boundary conditions, providing sharp asymptotic estimates for
solutions near the Dirichlet-Neumann junction and, as a consequence, unique continuation
properties. In addition, comparing our result with the aforementioned
papers,
here we also provide an estimate of the remainder term in the difference between the solution and its asymptotic profile.

Let $\O$ be an open subset of $\R^2$ with Lipschitz boundary.  Let
$\G_n\subset \de \O$ and $\G_d\subset \de \O$ be two nonconstant
curves (open in $\partial\Omega$) such that $\overline{\G_n} \cap \overline{\G_d}=\{P\}$ for some $P\in\partial\Omega$.
  We are interested in regularity of weak
solutions $u\in H^1(\O)$ to the mixed boundary value problem
\begin{equation}\label{eq:22}
  \begin{cases}
    -\Delta u= f(x)u,&\text{in }\Omega ,\\
\partial_\nu u=g(x)u,&\text{on }\Gamma_n,\\
 u=0,&\text{on }\Gamma_d,
  \end{cases}
\end{equation}
with $f\in L^\infty(\Omega )$ and $g\in C^1(\overline{\Gamma_n})$, see Section \ref{sec:auxiliary-problem} for the weak formulation.  Our aim is to
prove unique continuation properties from the Dirichlet-Neumann junction
$\{P\}= \overline{\G_n}\cap \overline{\G_d}$ and sharp asymptotics of nontrivial solutions
near $P$ provided $\de\O$ is of class $C^{2,\g}$ in a neighborhood of
$P$.
We mention that some regularity results for solutions to
  second-order elliptic problems with mixed Dirichlet--Neumann type
  boundary conditions were obtained in \cite{Kassmann,Savare}, see
  also the references therein.

Some interest  in the derivation of asymptotic expansions for
solutions to planar mixed boundary value problems at Dirichlet-Neumann
junctions arises  in the study of crack problems, see e.g. \cite{dmot,lt}. Indeed, if we
consider an  elliptic
equation in a planar domain with a crack and prescribe Neumann
conditions on the crack and Dirichlet conditions on the rest of the
boundary, in the case of  the crack end-point belonging to the
boundary of the domain we are lead to consider a problem of the type
described above  in a neighborhood of
the crack's tip (which corresponds to the Dirichlet-Neumann
junction). We recall (see
e.g. \cite{dmot})
that, in crack
problems, the coefficients of the asymptotic expansion of solutions near the crack's
tip  are related to the so called \emph{stress intensity factor}.\\

In order to get a precise asymptotic expansion of $u$ at point
$P\in \overline{\G_n}\cap \overline{\G_d}$, we will need to assume that $\de \O$ is of class
$C^{2,\d}$ near $P$. The asymptotic profile of the solution will be
given by the function
\be \label{eq:def-F-k}
F_k(r\cos \th,r\sin\th)=r^{\frac{2k-1}{2}}
\cos\left(\frac{2k-1}{2}\th\right), \quad r>0,\ \th\in (0,\pi),
\ee
 for some $k\in \N\setminus\{0\}$. We note that    $F_{k}\in H^1_{\rm loc}(\R^2)$ and solves the equation
\begin{align}\label{eq:eq-for-F-k}
\begin{cases}
\D F_k =0,& \textrm{in $\R^2_+ $},\\
F_k(x_1,0)=0,& \textrm{ for  $x_1<0$},\\
\de_{x_2} F_k(x_1,0)=0,& \textrm{ for  $x_1>0$},
\end{cases}
\end{align}
where here and in the following $\R^2_+:=\{(x_1,x_2)\in \R^2\,:\, x_2>0\}$.

The main result of the present paper provides an evaluation of the
behavior of weak solutions $u\in H^1(\O)$ to \eqref{eq:22} at the
boundary point where the boundary conditions change.
In order to simplify the statement and without losing generality, we
can fix the cartesian axes in such a way that the following assumptions
on $\Omega\subset\R^2$ are satisfied.  Here and in the remaining of
this paper, $\G_n,\G_d\subset\de\O$ are
nonconstant curves (open as subsets of $\partial\Omega$) such that  
$\overline{\G_n}\cap \overline{\G_d}=\{0\}$ with $0\in\partial\Omega$.
 \begin{enumerate}
 \item[(i)]  The domain $\O$ is of class $C^{2,\d}$ in a neighborhood
   of $0$, for some $\d>0$.
 \item[(ii)] The unit vector
 $e_1:=(1,0)$ is tangent to $\partial \Omega$ at $0$ and pointed towards $\G_n$. Moreover,  the exterior unit normal vector to $\partial\Omega$ at $0$ is $(0,-1)$.
\end{enumerate}

We are now in position to state the main result of the present paper.
\begin{Theorem}\label{t:main-u}
  We assume that $\O$ satisfies the assumption (i)-(ii) above.  Let
  $u\in H^1(\O)$ be a nontrivial weak solution to (\ref{eq:22}), with
  $f\in L^\infty(\Omega )$ and $g\in C^1(\overline{\Gamma_n})$.  Then,
  there exist $k_0\in \N\setminus\{0\}$, $\b \in\R\setminus\{0\} $ and
  $r>0$ such that, for every $\varrho\in (0,1/2)$, there exists $C>0$
  such that
  \be \label{eq:asympt-u-sharp}
 |u(x)- \b F_{k_0}(\vp(x))|\leq C |x|^{
    \frac{2k_0-1}{2}+\varrho}, \quad\textrm{for every
    $x\in \Omega\cap \ov{B_r^+}$}.  
\ee
  Here, the function $\phi: \ov{\O\cap B_{r_0}}\to \ov{\R^2_+}$ is a conformal map  of class $C^2$, for some $r_0>0$ only depending on $ \O$.
\end{Theorem}

  \begin{remark}\label{rem:conf}
Here and in the sequel, we identify $\R^2$ with the
  complex plane $\C$; hence, by a \emph{conformal map} on an open set
  $U\subset\R^2$ we mean a holomorphic function with complex
  derivative everywhere non-zero on $U$. We notice that, if $\Omega$
  satisfies (i)-(ii) and
  $\phi: \ov{\O\cap B_{r_0}}\to \ov{\R^2_+}$ is conformal, then
  $D\varphi(0)=\alpha \mathop{\rm Id}$ and $\varphi'(0)=\alpha$ for
  some real $\alpha>0$, where $D\varphi$ denotes the jacobian matrix
  of $\varphi$ and  $\varphi'$ denotes the complex
  derivative of $\varphi$.
\end{remark}


As a direct consequence of Theorem \ref{t:main-u}, we derive the
following  Hopf-type lemma.
\begin{Corollary}\label{c:upper_bound}
 Under the same assumptions as in Theorem \ref{t:main-u},
   let $u\in H^1(\O)$ be a non-trivial weak solution to
  (\ref{eq:22}), with  $u\geq 0$. Then
    \begin{enumerate}[\rm (i)]
    \item for every $t\in[0,\pi)$,
\[
\lim_{r\to 0}\frac{u(r\cos t,r\sin
  t)}{r^{1/2}}=\beta\alpha^{1/2}\cos\left(\frac t2\right)>0,
\]
where $\alpha=\varphi'(0)>0$ and $\varphi$ is as in Theorem \ref{t:main-u};
\item  for every   cone $\cC\subset \R ^2$ satisfying $(1,0)\in \cC$ and $(-1,0)\in \R^2\setminus \ov{\cC}$, we have
\be
{\mathop{\liminf}_{\substack{x\to 0\\x\in \ov{\O}\cap \cC}}} \frac{u(x)}{|x|^{1/2}}>0.
\ee
    \end{enumerate}
\end{Corollary}

A further relevant byproduct of our asymptotic analysis is the
following unique continuation principle, whose proof follows
directly from Theorem \ref{t:main-u}.

\begin{Corollary}\label{c:unique_continuation}
Under the same assumptions as in Theorem \ref{t:main-u}, let $u\in H^1(\O)$ be a weak solution to
  (\ref{eq:22}) such that $u(x)=O(|x|^n)$ as $x\in\O$,
  $|x|\to 0$, for any $n\in \N$. Then $u\equiv 0$.
\end{Corollary}

We observe that 
Theorem \ref{t:main-u} provides a sharp asymptotic expansion (and
consequently a unique continuation principle) at the boundary for
$\frac12$-fractional elliptic equations in dimension 1. Indeed, if $v\in
H^{1/2}(\R)$ weakly solves 
\[
\begin{cases}
(-\Delta)^{1/2}v=g(x)v,&\text{in }(0,R),\\
v=0,&\text{in }\R\setminus(0,R),
\end{cases}
\]
for some $g\in C^1([0,R])$, then its harmonic extension $V\in H^1_{\rm loc}(\overline{\R^2_+})$ weakly solves 
\begin{equation}\label{eq:32}
\begin{cases}
-\Delta V=0,&\text{in }\R^2_+,\\
\partial_\nu V=g(x)V,&\text{on }(0,R)\times\{0\},\\
V=0,&\text{on }(\R\setminus(0,R))\times\{0\},
\end{cases}
\end{equation}
 see \cite{cs}.
Theorem \ref{t:main-u} and Corollary \ref{c:unique_continuation}
apply
 to \eqref{eq:32}. Hence,  $V$
(and in particular  its restriction $v$) satisfies  expansion \eqref{eq:asympt-u-sharp}
 and a strong unique
 continuation principle from  $0$ (i.e. from
 a boundary point of the domain of $v$). We mention that unique
 continuation principles from interior points for fractional elliptic
 equations were established in \cite{FallFelli}. 

We do not know if  the $C^{2,\d}$ regularity on $\O$ and $C^1$
regularity of the boundary potential $g$ in Theorem \ref{t:main-u}
can be weakened in order to obtain a unique continuation property.
On the other hand, we can conclude that  a regularity assumption on the boundary is
crucial for excluding the presence of
logarithms in the asymptotic expansion at the junction.
Indeed, in Section \ref{sec:examples} we produce an example
of a harmonic function on a domain with a $C^1$-boundary which is not of
class $C^{2,\delta}$,
satisfying null Dirichlet boundary conditions on a portion of the
boundary and  null Neumann boundary conditions on the other portion, but exhibiting dominant
logarithmic terms in its asymptotic expansion.

The proof of Theorem \ref{t:main-u} combines the use of an Almgren
type monotonicity formula, blow-up analysis and sharp regularity
estimates. Indeed regularity estimates yield the expansion of $u$
near zero as follows:
\begin{equation}\label{eq:Taylor-u-intro}
\left\|u-\sum_{k=1}^{k_0} a_k(r) F_k\circ\vp \right\|_{L^\infty(B_r)}
\leq C r^{\frac{2k_0-1}{2}+\varrho},
\end{equation}
 for every $\varrho\in (0,1/2)$,
for some $C>0$, $k_0\geq 1$ and where
$a_k=\frac{\la u,F_k\circ\vp \ra_{L^2(B_r)}}{\|F_j\circ\vp
  \|_{L^2(B_r) }^2}$.
Now, if $u$ is nontrivial, a blow-up analysis combined with Almgren type
monotonicity formula allows to depict a $k_0\geq 1$ for which
$ a_{k_0}(r)\to \b\not=0$ and $a_k(r)\to 0$ for every $k< k_0$ as $r\to0$.
The proof of \eqref{eq:Taylor-u-intro} uses also a blow-up analysis
argument inspired by   Serra \cite{Serra}, see also \cite{RS2,RS3}.

The paper is organized as follows. In Section
\ref{sec:auxiliary-problem} we introduce an auxiliary
equivalent problem  obtained by a conformal
diffeomorphic deformation straightening $B_1\cap\partial\Omega$ near
$0$ and state Theorem \ref{t:asym} giving the sharp asymptotic
behaviour of its solutions. Section \ref{sec:hardy-poincare-type}
contains some Hardy-Poincar\'e type inequalities for $H^1$-functions
vanishing on a portion of the boundary of half-balls. In Section
\ref{sec:monotonicity-formula} we develop an Almgren type monotonicity
formula for the auxiliary problem which yields good energy estimates
for rescaled solutions thus allowing the fine blow-up analysis
performed in Section  \ref{sec:blow-up-analysis} and hence the proof
of Theorem \ref{t:asym}. Section \ref{sec:asymptotics-u} contains the
proof of the main Theorem \ref{t:main-u}, which is based on Theorem
\ref{t:asym} and on some regularity and
  approximation results established in Section
  \ref{sec:some-regul}. Finally, Section
 \ref{sec:examples} is devoted to the construction of an example of a
 solution with logarithmic dominant term in a domain violating the
 $C^{2,\delta}$-regularity assumption.

\section{The auxiliary problem}\label{sec:auxiliary-problem}
For every $R>0$  let $B_R=\{(x_1,x_2)\in\R^2:x_1^2+x_2^2<R^2\}$ and
$B_R^+=\{(x_1,x_2)\in B_R:x_2>0\}$.  Since $\de \O$ is of class
$C^{2,\d}$ near zero, we can find $r_0>0$ such that
$\G:=\de \O\cap B_{ r_0}$ is a $C^{2,\d}$ curve.  Here and in the
following, we let $\cB$ be a $C^{2,\d} $ simply connected open bounded set such
that $ \cB\subset \O$ and $\partial\mathcal B\cap\partial\Omega=\Gamma$.
For some  functions
\begin{equation}\label{eq:27}
 f\in
L^\infty(  \cB) \quad\text{and}\quad g\in C^1(\overline{\Gamma_n}),
\end{equation}
let $u\in
H^1( \cB)$ be  a solution to
\begin{equation}\label{eq:pb}
  \begin{cases}
    -\Delta u= f(x)u,&\text{in }\mathcal  \cB,\\
\partial_\nu u=g(x)u,&\text{on }\Gamma_n ,\\
 u=0,&\text{on }\Gamma_d .
  \end{cases}
\end{equation}
We introduce the space $H^1_{0,\Gamma_d}(\mathcal
B)$ as the closure in $H^1(\mathcal B)$ of the subspace
$$  C^\infty_{0,\Gamma_d}(\mathcal
B):=  \{u\in C^\infty(\overline{\mathcal B}):u=0\text{ on }\Gamma_d\cap \de \cB\}.$$ 
We say that $u\in H^1(\mathcal B)$ is a weak solution to
\eqref{eq:pb} if
\begin{equation*}
\begin{cases}
u\in H^1_{0,\Gamma_d}(\mathcal B),\\[8pt]
 {\displaystyle{\int_{\mathcal B} \nabla u(x) \nabla v(x) \,
dx=\int_{\mathcal B} f(x)u(x)v(x)\, dx+\!\!\int_{\Gamma_n}
guv \, ds }}\quad \text{for any $v\in
C^\infty_{0,\partial\mathcal B\setminus\Gamma_n}(\mathcal
B) $}
\end{cases}
\end{equation*}
where
$C^\infty_{0,\partial\mathcal B\setminus\Gamma_n}(\mathcal
B)=
\{u\in C^\infty(\overline{\mathcal B}):u=0\text{ on
}\partial\mathcal B\setminus\Gamma_n\}$.
Since $\cB$ is of class $C^{2,\d}$, in view of the Riemann mapping Theorem and \cite[Theorem
5.2.4]{Krantz}, there exists a  conformal map  $\hat \vp :\ov{ \cB}
\to  \ov B_1$
which is of class $C^2$. Let $N=\hat\vp(0)\in \de B_1$ and let $S$ be its antipodal. We then consider the map
$\ti \vp : \R^2\setminus\{S\} \to \R^2\setminus\{\overline S\}$ given by
$\ti\vp (z):= 2\frac{\overline{z-S}}{|z-S|^2}+\overline{S}$, where,
for every $z\in\R^2\simeq\C$, $\overline{z}$ denotes the complex
conjugate of $z$. This map is conformal and $\ti\vp(N)=0$. In
addition $\ti \vp( \ov {B_1} \setminus\{S\}) \subset \ov{ \cP}$
where $\mathcal P$ is the half plane not containing $\overline S$ whose boundary is the line passing through the origin orthogonal to $\overline S$.

Then the map $ \ti \vp \circ \hat \vp$ is a conformal
map which is of class $C^2$ from  a neighborhood of the origin
$\overline{{\mathcal B}\cap B_r}$ into $\ov{ \cP}$ for some $r>0$.  It is now clear that there
exists a rotation $\cR$ and a real number $R>0$ such that, letting
$\cU_R:=\vp^{-1}(B_R^+)$, the map
$\vp:=\cR \circ \ti \vp \circ \hat \vp : \ov{\cU_R}\to \ov{B_R^+}$ is
an  invertible
conformal map of class $C^2$  with inverse
  $\vp^{-1}:\ov{B_R^+}\to \ov{\cU_R}$ of class $C^2$. Moreover $\varphi(0)=0$.

 Since $\vp$ is a conformal diffeomorphism, in view of Remark \ref{rem:conf} we have that, under the
  assumptions of Theorem \ref{t:main-u},
\begin{equation}\label{eq:20}
\mathop D\phi(0)
=\alpha
\mathop{\rm
    Id},\quad\text{with }{\alpha= \phi'(0)>0},
\end{equation}
being $\phi'(0)$ the complex derivative of $\phi$ at $0$, which turns
out to be real because of the assumption that $(1,0)$ is tangent to
$\partial\Omega$ at $0$ and strictly positive because of the
assumption that  the exterior unit normal vector to $\partial\Omega$ at
  $0$ is $(0,-1)$.
 In addition, \eqref{eq:20} implies that, if $R$ is
   chosen sufficiently small,
$\vp ^{-1}((-R,0) \times\{0\})\subset\G_d$ and
$\vp ^{-1}((0,R)\times\{0\})\subset \G_n$.

Therefore letting  $w=u\circ \phi^{-1}: B_R^+ \to  \R $ and $\Psi:= \vp^{-1}$,
we then have that $w\in H^1(B_R^+)$ solves
\begin{equation}\label{eq:5}
\begin{cases}
  -\Delta   w(z) =p(z)w(z) ,&\text{in }B_R^+,\\
\partial_\nu w(x_1,0) =q(x_1) w(x_1,0) ,& x_1\in (0,R),\\
 w=0,&\text{on }(-R,0)\times\{0\} ,
  \end{cases}
\end{equation}
with
\[
p(z)=|\Psi'(z)|^2f(\Psi(z)) ,\qquad \qquad
q(x_1)=(g(\Psi(x_1,0))|{\Psi'(x_1,0)}|.
\]
It is plain that $ p\in L^\infty(B_R^+)$ and $q \in C^1([0,R))$.
Here and in the following, for every $r>0 $, we define
\be\label{eq:def:Gr-Gd}
\G_n^r:=   (0,r)\times\{0\}\qquad\textrm{ and } \qquad   \G_d^r:=(-r,0)\times\{0\}.
\ee

The following theorem describes the behaviour of $w$ at $0$ in
terms of the limit of the Almgren quotient associated to $w$,
which is defined as
\[
\mathcal N(r)=\frac{ \int_{B_r^+} |\nabla w|^2dz
  -
\int_{B_r^+} pw^2dz
 -\int_0^r q(x)w^2(x,0)\,dx
}{\int_0^\pi w^2(r\cos t,r\sin t)\,dt}.
\]
In Section \ref{sec:monotonicity-formula} we will prove that $\mathcal
N$ is well defined in the interval $(0,R_0)$ for some $R_0>0$.

\begin{Theorem} \label{t:asym}
Let $w$ be a nontrivial  solution to \eqref{eq:5}. Then there
exists $k_0\in \N$, $k_0\geq1$, such that
\begin{equation}\label{eq:35}
\lim_{r\to 0^+}\mathcal N(r)=\frac{2k_0-1}2.
\end{equation}
Furthermore
\begin{equation*}
\tau^{-\frac{2k_0-1}2} w(\tau z)\to \beta |z|^{\frac{2k_0-1}2}
\cos\bigg(\tfrac{2k_0-1}{2}\mathop{\rm Arg} z\bigg)\quad\text{as
}\tau\to 0^+
\end{equation*}
 strongly in $H^1(B_r^+)$ for all $r>0$  and in  $C^{0,\mu}_{\rm
    loc}(\overline{\R^2_+}\setminus\{0\})$  for every $\mu\in
  (0,1)$,
where $\beta\neq0$ and
\begin{align}\label{eq:21}
    \beta=&
\frac2\pi\int_{0}^\pi R^{-\frac{2k_0-1}{2}} w(R\cos s,R\sin s)
           \cos\left(\tfrac{2k_0-1}{2}s\right)\,ds\\
\notag& +\frac2\pi
\!\int_0^\pi\!\bigg[\int_0^R\tfrac{t^{-k_0+3/2}-R^{1-2k_0}
t^{k_0+1/2}}{2{k_0}-1} p(t\cos s,t\sin s) w(t\cos s,t\sin s)
\,dt \bigg]\cos\left(\tfrac{2k_0-1}{2}s\right)\,ds\\
\notag& +\frac2\pi\int_0^R\frac{t^{1/2-k_0}-R^{1-2k_0}
t^{k_0-1/2}}{2{k_0}-1} q(t) w(t,0)\,dt.
\end{align}
 In particular
\begin{equation}\label{eq:shar-asym-w-tau}
\tau^{-\frac{2k_0-1}2}w(\tau\cos t,\tau\sin t)\to \beta
\cos\left(\tfrac{2k_0-1}{2}t\right)\quad \text{in }
C^{0,\mu}([0,\pi])  \quad \text{as }\tau\to 0^+.
\end{equation}
\end{Theorem}
The proof of Theorem \ref{t:asym} is based on the study of the
monotonicity properties of the Almgren function $\mathcal N$ and on a
fine blow-up analysis which will be performed in Sections
\ref{sec:monotonicity-formula} and  \ref{sec:blow-up-analysis}.

\section{Hardy-Poincar\'e type inequalities}\label{sec:hardy-poincare-type}

In the description of the asymptotic behavior at the Dirichlet-Neumann
junction of solutions to
equation \eqref{eq:5}
a crucial role is played by eigenvalues and eigenfunctions of
the angular component of the principal part of the operator.

Let us
consider the eigenvalue problem
\begin{equation}\label{eq:2}
  \begin{cases}
    -\psi''=\lambda \psi,&\text{in }[0,\pi],\\
\psi'(0)=0,\\
\psi(\pi)=0.
  \end{cases}
\end{equation}
It is easy to verify that \eqref{eq:2}
admits the sequence of (all simple) eigenvalues
\[
\lambda_k=\frac14(2k-1)^2,\quad k\in \N,\ k\geq1,
\]
with corresponding eigenfunctions
\[
\psi_k(t)=\cos\left(\tfrac{2k-1}{2}t\right),\quad k\in \N,\
k\geq1.
\]
It is well known that the normalized eigenfunctions
\begin{equation}\label{eq:17}
\left\{\sqrt{\tfrac2\pi}\cos\left(\tfrac{2k-1}{2}t\right)\right\}_{k\geq1}
\end{equation}
  form an orthonormal basis of the space $L^2(0,\pi)$.
Furthermore, the first eigenvalue $\lambda_1=\frac14$ can be
characterized as
\begin{equation}\label{eq:3}
  \lambda_1=\frac14=\min_{\substack{\psi\in
      H^1(0,\pi)\setminus\{0\}\\\psi(\pi)=0}}
  \frac{\int_0^\pi |\psi'(t)|^2\,dt}{\int_0^\pi |\psi(t)|^2\,dt}.
\end{equation}
For every $r>0$, we let (recall \eqref{eq:def:Gr-Gd} for the definition of $\G_d^r$)
\[
\mathcal H_r= \{w\in H^1(B_r^+): w=0\text{ on $\G_d^r$ } \}.
\]
As a consequence of \eqref{eq:3} we obtain the following
Hardy-Poincar\'e inequality in $\mathcal H_r$.

\begin{Lemma}\label{l:hardy_poincare}
  For every $r>0$ and  $w\in \mathcal H_r$, we have that
\[
\int_{B_r^+}|\nabla w(z)|^2\,dz\geq \frac 14
\int_{B_r^+}\frac{|w(z)|^2}{|z|^2}\,dz.
\]
\end{Lemma}
\begin{proof}
  Let $w\in C^\infty(\overline{B_r^+})$ with $w=0$ on
  $\ov{ \G_d^r}=[-r,0]\times\{0\}$. Then, in view of \eqref{eq:3},
  \begin{align*}
    \int_{B_r^+}&|\nabla w(z)|^2\,dz\\
&=\int_0^r\int_0^\pi
    \rho\left(\left|\tfrac{\partial}{\partial \rho}(w(\rho\cos
        t,\rho\sin t))\right|^2+\frac1{\rho^2}\left|\tfrac{\partial}{\partial t}(w(\rho\cos
        t,\rho\sin t))\right|^2\right)\,dt\,d\rho\\
&\geq \int_0^r \frac1{\rho}\left(\int_0^\pi
    \left|\tfrac{\partial}{\partial t}(w(\rho\cos
        t,\rho\sin t))\right|^2dt\right)\,d\rho\\
&\geq \frac14 \int_0^r \frac1{\rho}\left(\int_0^\pi
    |w(\rho\cos
        t,\rho\sin t)|^2dt\right)\,d\rho=\frac 14
\int_{B_r^+}\frac{|w(z)|^2}{|z|^2}\,dz
  \end{align*}
We conclude by density, recalling that the space of smooth
functions vanishing on $[-r,0]\times\{0\}$ is dense in $\mathcal
H_r$, see e.g.  \cite{Doktor2006}.
\end{proof}

\begin{Lemma}\label{l:hardy_boundary}
  For every $r>0$ and  $w\in \mathcal H_r$, we have that $x_1^{-1}w^2(x_1,0)\in L^1(0,r)$ and
\[
\int_0^r\frac{w^2(x_1,0)}{x_1}\,dx_1\leq \pi \int_{B_r^+}|\nabla
w(z)|^2\,dz.
\]
\end{Lemma}
\begin{proof}
  Let $w\in C^\infty(\overline{B_r^+})$ with $w=0$ on
  $[-r,0]\times\{0\}$. Then for any $0<x_1<r$
  \begin{align*}
    |w(x_1,0)|&=\left|\int_0^\pi\frac{d}{dt}w(x_1\cos t,x_1\sin t)\, dt\right|=
\left|\int_0^\pi x_1\nabla w(x_1\cos t,x_1\sin t) \cdot(-\sin
t,\cos t)\,
  dt\right|\\
&\leq \sqrt \pi\sqrt{\int_0^\pi x_1^2|\nabla w(x_1\cos t,x_1\sin
t)|^2\,dt }.
  \end{align*}
It follows that
\begin{align*}
  \int_0^r\frac{w^2(x_1,0)}{x_1}\,dx_1 \leq
  \pi\int_0^r\int_0^\pi x_1|\nabla w(x_1\cos t,x_1\sin t)|^2\,dt\,dx_1=\pi \int_{B_r^+}|\nabla
w(z)|^2\,dz.
\end{align*}
We conclude by density.
\end{proof}

\section{The monotonicity formula}\label{sec:monotonicity-formula}
Let $w\in H^1(B_R^+)$ be a non trivial solution to \eqref{eq:5}.
For every $r\in (0,R]$ we define
\begin{equation}\label{D(r)}
  D(r)=
 \int_{B_r^+} |\nabla w|^2dz
  -
\int_{B_r^+} pw^2dz
 -\int_0^r q(x_1)w^2(x_1,0)\,dx_1
\end{equation}
and
\begin{equation} \label{H(r)}
H(r)=\frac{1}{r}\int_{S_r^+}w^2 \, ds = \int_0^\pi w^2(r\cos
t,r\sin t)\,dt,
\end{equation}
where $S_r^+:=\{(x_1,x_2):x_1^2+x_2^2= r^2\text{ and
}x_2>0\}$.

In order to differentiate the functions $D$ and $H$, the following
Pohozaev type identity is needed.

\begin{Theorem} \label{t:pohozaev}
Let $w$ solve \eqref{eq:5}. Then for a.e. $r\in (0,R)$ we have
\begin{multline}\label{eq:poho}
\frac{r}{2}\int_{S_r^+} |\nabla w|^2ds
=r\int_{S_r^+}\bigg|\frac{\partial w}{\partial
    \nu}\bigg|^2\,ds\\
  -\frac{1}2\int_{0}^r\big(q(x_1)+ x_1
  q'(x_1)\big)w^2(x_1,0)\,dx_1
 +
  \frac{r}2q(r)w^2(r,0)
+\int_{B_r^+}pwz\cdot\nabla w\,dz
\end{multline}
and
\begin{equation}\label{eq:poho2}
 \int_{B_r^+} |\nabla w|^2dz
  =
\int_{B_r^+} pw^2dz
 +\int_{S_r^+}\frac{\partial w}{\partial\nu}w\,ds+
\int_0^r q(x_1)w^2(x_1,0)\,dx_1.
\end{equation}
\end{Theorem}

\begin{proof}
We observe that, by elliptic regularity theory, $w\in
H^2(B_r^+\setminus  B_\e^+)$ for all $0<\e<r<R$. Furthermore, the
fact that $w$ has null trace on $\Gamma_d^R$ implies that
$\frac{\partial w}{\partial x_1}$ has null trace on $\Gamma_d^R$.
Then, testing \eqref{eq:5} with $z\cdot\nabla w$ and integrating
over $B_r^+\setminus  B_\e^+$, we obtain that
\begin{multline}\label{eq:6}
  \frac r2\int_{S_r^+}|\nabla w|^2ds-  \frac \e 2\int_{S_\e^+}|\nabla
  w|^2ds=
\int_{B_r^+\setminus  B_\e^+}pwz\cdot\nabla w\,dz\\
+r\int_{S_r^+}\bigg|\frac{\partial w}{\partial
    \nu}\bigg|^2\,ds-
\e\int_{S_\e^+}\bigg|\frac{\partial w}{\partial
    \nu}\bigg|^2\,ds+\int_\e^r q(x_1)w(x_1,0) x_1\frac{\partial
    w}{\partial x_1}(x_1,0)\,dx_1.
\end{multline}
An integration by parts, which can be easily justified by an
approximation argument, yields that
\begin{multline}\label{eq:1}
\int_\e^r q(x_1)w(x_1,0) x_1\frac{\partial
    w}{\partial x_1}(x_1,0)\,dx_1=\frac r2 q(r)w^2(r,0)\\ -\frac \e2
  q(\e)w^2(\e,0)-\frac12\int_\e^r (q+x_1q')w^2(x_1,0)\,dx_1.
\end{multline}
We observe that there exists a sequence $\e_n\to 0^+$ such that
\[
\lim_{n\to \infty} \left[  \e_n w^2(\e_n,0)+
 \e_n\int_{S_{\e_n}^+}|\nabla
  w|^2ds \right]=0.
\]
Indeed, if no such sequence exists, there would exist $\e_0>0$
such that
$$
  w^2(r,0)+
\int_{S_{r}^+}|\nabla
  w|^2ds  \geq \frac{C}{r}\quad \text{for all } r\in(0,\e_0),\quad\text{for
    some }C>0;
$$
integration of  the above inequality on $(0,\e_0)$ would then
contradict the fact that $w\in H^1(B_{R}^+)$ and, by trace
embedding, $w\in L^2(\Gamma_n^{\e_0})$. Then, passing to the limit
in \eqref{eq:6} and \eqref{eq:1} with $\e=\e_n$ yields
\eqref{eq:poho}. Finally
 \eqref{eq:poho2} follows by testing \eqref{eq:5} with $w$ and
 integrating by parts in $B_r^+$.
\end{proof}

In the following lemma we compute the derivative of the function $H$.
\begin{Lemma} \label{l:hprime}
$H\in W^{1,1}_{\rm loc}(0,R)$ and
\begin{equation}\label{H'}
  H'(r)=2
\int_0^\pi w(r\cos t,r\sin t) \tfrac{\partial
w}{\partial\nu}(r\cos t,r\sin t)   \,dt= \frac2r\int_{S_r^+}w
\tfrac{\partial w}{\partial\nu}\,ds,
\end{equation}
in a distributional sense and for a.e. $r\in (0, R)$, and
\begin{equation}\label{H'2}
 H'(r)=\frac2r D(r), \quad \text{for a.e. } r\in (0, R).
\end{equation}
\end{Lemma}

\begin{proof}
Let $\phi\in C^\infty_{\rm c}(0,R)$. Since $w,\nabla w\in
L^2(B_R^+)$ and $w\in C^1(B_R^+)$, using twice Fubini's Theorem we
obtain that
\begin{align*}
  &\int_0^R H(r)\phi'(r)\,dr=
  \int_0^R \bigg(\int_0^\pi w^2(r\cos t,r\sin
  t)\,dt\bigg)\phi'(r)\,dr\\
&=
  \int_0^\pi \bigg(\int_0^R w^2(r\cos t,r\sin t)
  \phi'(r)\,dr\bigg)\,dt
= -\int_0^\pi \bigg(\int_0^R \frac{d}{dr}\Big(w^2(r\cos t,r\sin
t)\Big)
  \phi(r)\,dr\bigg)\,dt\\
&= -\int_0^\pi \bigg(\int_0^R \Big(2w(r\cos t,r\sin t)
\tfrac{\partial w}{\partial\nu}(r\cos t,r\sin t) \Big)
  \phi(r)\,dr\bigg)\,dt\\
&= -\int_0^R \bigg(\int_0^\pi \Big(2w(r\cos t,r\sin t)
\tfrac{\partial w}{\partial\nu}(r\cos t,r\sin t) \Big)
  \,dt\bigg) \phi(r)\,dr
\end{align*}
thus proving \eqref{H'}. Identity \eqref{H'2} follows directly
from \eqref{H'} and \eqref{eq:poho2}.
\end{proof}

Let us now study the regularity of the function $D$.
\begin{Lemma}\label{l:dprime}
  The function $D$ defined in \eqref{D(r)} belongs to $W^{1,1}(0, R)$ and
\begin{align}\label{D'F}
  D'(r)&=
2\int_{S_r^+} \bigg|\frac{\partial w}{\partial
    \nu}\bigg|^2ds\\
\notag&\quad -\frac1r\int_{0}^r\big(q(x_1)+ x_1
q'(x_1)\big)w^2(x_1,0)\,dx_1+\frac2r \int_{B_r^+}pwz\cdot\nabla
w\,dz-\int_{S_r^+}pw^2\,ds
\end{align}
in a distributional sense and for a.e. $r\in (0,R)$.
\end{Lemma}

\begin{proof}
From the fact that $w\in  H^1(B_R^+)$ and $w\big|_{\Gamma_n^R}\in
L^2(\Gamma_n^R)$, we deduce that $D$ belongs to $W^{1,1}(0,R)$ and
\begin{equation} \label{I'(r)}
  D'(r)=
 \int_{S_r^+} |\nabla w|^2ds
  -
\int_{S_r^+} pw^2ds
 - q(r)w^2(r,0)
\end{equation}
 for a.e. $r\in (0,R)$ and in the distributional sense.

The conclusion follows combining \eqref{I'(r)} and~\eqref{eq:poho}.
\end{proof}

 \begin{Lemma} \label{welld}
There exists $R_0\in(0,R)$   such that $H(r)>0$ for any $r\in
(0,R_0)$.
\end{Lemma}

\begin{proof}
Let $R_0\in (0,R)$ be such that
\begin{align}\label{eq:r_0}
  4\|p\|_{L^\infty(B_R^+)}R_0^2+\pi\|q\|_{L^\infty(\Gamma_n^R)}R_0<1.
\end{align}
Assume by contradiction that there exists $r_0\in(0,R_0)$ such
that $H(r_0)=0$, so that $w=0$ a.e. on $S_{r_0}^+$. From
\eqref{eq:poho2} it follows that
\begin{equation*}
 \int_{B_{r_0}^+} |\nabla w|^2dz
  -
\int_{B_{r_0}^+} pw^2dz
 -\int_0^{r_0} q(x_1)w^2(x_1,0)\,dx_1=0.
\end{equation*}
From Lemmas \ref{l:hardy_poincare} and \ref{l:hardy_boundary}, we get
\begin{align*}
0&= \int_{B_{r_0}^+} |\nabla w|^2dz
  -
\int_{B_{r_0}^+} pw^2dz
 -\int_0^{r_0} q(x_1)w^2(x_1,0)\,dx_1\\
\notag & \geq
  \Big[1-
4\|p\|_{L^\infty(B_R^+)}r_0^2-\pi\|q\|_{L^\infty(\Gamma_n^R)}r_0\Big]
\int_{B_{r_0}^+}|\nabla w|^2dz,
\end{align*}
which, together with \eqref{eq:r_0} and Lemma
\ref{l:hardy_poincare}, implies $w\equiv 0$ in $B_{r_0}^+$.  From
classical unique continuation principles for second order elliptic
equations with locally bounded coefficients (see e.g.
\cite{wolff}) we can conclude that $w=0$ a.e. in $B_R^+$, a
contradiction.
\end{proof}

Thanks to Lemma \ref{welld}, the frequency function
\begin{equation}\label{N(r)}
\mathcal N:(0,R_0)\to\R,\quad \mathcal N(r)=\frac{D(r)}{H(r)},
\end{equation}
is well defined.
 Using Lemmas  \ref{l:hprime} and \ref{l:dprime}, we  now
 compute the derivative of ${\mathcal N}$.

 \begin{Lemma}\label{mono} The function
   ${\mathcal N}$ defined in \eqref{N(r)} belongs to $W^{1,1}_{{\rm
       loc}}(0, R_0)$ and
\begin{align}\label{formulona}
{\mathcal N}'(r)=\nu_1(r)+\nu_2(r)
\end{align}
in a distributional sense and for a.e. $r\in (0,R_0)$, where
\begin{align}\label{eq:nu1}
\nu_1(r)=&\frac{2r\Big[
    \left(\int_{S_r^+}
  \left|\frac{\partial w}{\partial \nu}\right|^2 ds\right) \cdot
    \left(\int_{S_r^+}
  w^2\,ds\right)-\left(
\int_{S_r^+}
  w\frac{\partial w}{\partial \nu}\, ds\right)^{\!2} \Big]}
{\left( \int_{S_r^+}
  w^2\,ds\right)^2}
\end{align}
and
\begin{align}\label{eq:nu2}
  \nu_2(r)= &
-\frac{\int_{0}^r\big(q(x)+ x q'(x)\big)w^2(x,0)\,dx}{\int_{S_r^+}
  w^2\,ds}+2\frac{\int_{B_r^+}pwz\cdot\nabla w\,dz}{\int_{S_r^+}w^2\,ds}
-\frac{r\int_{S_r^+}pw^2\,ds}{\int_{S_r^+}w^2ds}.
\end{align}
\end{Lemma}

\begin{proof} From Lemmas \ref{l:hprime}, \ref{welld}, and
  \ref{l:dprime}, it follows that ${\mathcal N}\in W^{1,1}_{{\rm
      loc}}(0,R_0)$. From
\eqref{H'2} we deduce that
$$
{\mathcal N}'(r)=\frac{D'(r)H(r)-D(r)H'(r)}{(H(r))^2}
=\frac{D'(r)H(r)-\frac{1}{2} r (H'(r))^2}{(H(r))^2}
$$
and the proof of the lemma easily follows from (\ref{H'}) and
(\ref{D'F}).
\end{proof}

\noindent We now prove that ${\mathcal N}(r)$ admits a finite
limit as $r\to 0^+$.

\begin{Lemma}\label{l:limitN}
  There exists $\gamma\in[0,+\infty)$ such that $\lim_{r\to 0^+}\mathcal N(r)=\gamma$.
\end{Lemma}
\begin{proof}
  From Lemmas \ref{l:hardy_poincare} and \ref{l:hardy_boundary} it
  follows that
\begin{equation*}
    D(r)\geq
 \Big[1-
4\|p\|_{L^\infty(B_R^+)}r^2-\pi\|q\|_{L^\infty(\Gamma_n^R)}r\Big]
\int_{B_{r}^+}|\nabla w|^2dz,
\end{equation*}
hence  there exist $\bar r\in (0,R_0)$ and $C_1>0$ such that
\[
D(r)\geq C_1  \int_{B_{r}^+}|\nabla w|^2dz,\quad\text{for all
}r\in(0,\bar r).
\]
In particular
\begin{equation}\label{eq:4}
\mathcal N(r)\geq 0,\quad\text{for all }r\in(0,\bar r).
\end{equation}
Moreover, using again Lemmas \ref{l:hardy_poincare} and
\ref{l:hardy_boundary} we can estimate $\nu_2$ in $(0,\bar r)$ as
follows
\begin{align}\label{eq:7}
|\nu_2(r)|&\leq \frac{\|q+ x q'\|_{L^\infty(\Gamma_n^R)}\pi r
\int_{B_{r}^+}|\nabla w|^2dz}{\int_{S_r^+}
  w^2\,ds}\\
\notag&\quad
+\frac{\|p\|_{L^{\infty}(B_R^+)}r(1+4r^2)\int_{B_r^+}|\nabla
  w(z)|^2\,dz}{\int_{S_r^+}w^2\,ds}+r\|p\|_{L^{\infty}(B_R^+)}\\
\notag&\leq \frac1{C_1}\left(\|q+ x
q'\|_{L^\infty(\Gamma_n^R)}\pi+\|p\|_{L^{\infty}(B_R^+)}(1+4\bar
r^2)\right)\mathcal N(r) +\bar r\|p\|_{L^{\infty}(B_R^+)}.
\end{align}
Since $\nu_1\geq 0$ by Schwarz's inequality, from Lemma \ref{mono}
and the above estimate it follows that there exists $C_2>0$ such
that
\begin{equation}\label{eq:8}
\mathcal N'(r)\geq -C_2(\mathcal N(r)+1)\quad\text{for all
}r\in(0,\bar r),
\end{equation}
which implies that
\[
\frac{d}{dr}\left(e^{C_2r}(1+\mathcal N(r))\right)\geq0.
\]
It follows that  the limit of $r\mapsto e^{C_2r}(1+\mathcal
N(r))$ as $r\to0^+$  exists and is finite; hence the function
$\mathcal N$ has a finite limit $\gamma$ as $r\to 0^+$. From
\eqref{eq:4} we deduce that $\gamma\geq 0$.
\end{proof}

\noindent The function $H$ defined in \eqref{H(r)} can be
estimated as follows.
\begin{Lemma}\label{l:uppb}
  Let
  $\gamma:=\lim_{r\rightarrow 0^+} {\mathcal N}(r)$ be as in Lemma
  \ref{l:limitN}.  Then
\begin{equation} \label{1stest}
H(r)=O(r^{2\gamma})  \quad \text{as } r\to0^+.
\end{equation}
Moreover, for any $\sigma>0$,
\begin{equation} \label{2ndest}
 r^{2\gamma+\sigma} =O(H(r)) \quad \text{as } r\to0^+.
\end{equation}
\end{Lemma}

\begin{proof}
From Lemma \ref{l:limitN} we have that
\begin{equation}
  \label{eq:9}
\text{$\mathcal N$ is bounded in a neighborhood of $0$},
\end{equation}
hence from \eqref{eq:8} it follows that $\mathcal N'\geq -C_3$ for
some positive constant $C_3$  in a neighborhood of $0$. Then
\begin{equation} \label{qsopra}
{\mathcal N}(r)-\gamma=\int_0^r
    {\mathcal N}'(\rho) \, d\rho\geq -C_3 r
\end{equation}
 in a neighborhood of
$0$. From (\ref{H'2}), \eqref{N(r)}, and (\ref{qsopra}) we deduce
that,  in a neighborhood of $0$,
$$
\frac{H'(r)}{H(r)}=\frac{2\,{\mathcal N}(r)}{r}\geq
\frac{2\gamma}{r}-2C_3,
$$
which, after integration, yields (\ref{1stest}).

Since $\gamma=\lim_{r\rightarrow 0^+} {\mathcal N}(r)$, for any
$\sigma>0$ there exists $r_\sigma>0$ such that ${\mathcal
N}(r)<\gamma+\sigma/2$ for any $r\in (0,r_\sigma)$ and hence
$\frac{H'(r)}{H(r)}=\frac{2\,{\mathcal
    N}(r)}{r}<\frac{2\gamma+\sigma}{r}$ for all $r\in (0,r_\sigma)$.
By integration we obtain (\ref{2ndest}).
\end{proof}

\section{Blow-up analysis for the auxiliary problem}\label{sec:blow-up-analysis}
\begin{Lemma}\label{l:blowup}
 Let $w\in H^1(B_R^+)$ be a non trivial solution to \eqref{eq:5}.
 Let $\gamma:=\lim_{r\rightarrow 0^+} {\mathcal
    N}(r)$ be as in Lemma \ref{l:limitN}. Then
there exists $k_0\in \N$, $k_0\geq1$, such that
\[
\gamma=\frac{2k_0-1}{2}.
\]
Furthermore, for every sequence $\tau_n\to0^+$, there exist a
subsequence $\{\tau_{n_k}\}_{k\in\N}$   such that
\[
\frac{w(\tau_{n_k}z)}{\sqrt{H(\tau_{n_k})}}\to \widetilde w(z)
\]
 strongly in $H^1(B_r^+)$ and in  $C^{0,\mu}_{\rm
    loc}(\overline{B_r^+}\setminus\{0\})$  for every $\mu\in (0,1)$
  and all $r\in(0,1)$,
where
\begin{equation}\label{eq:11}
\widetilde w(r\cos t,r\sin t)= \pm\sqrt{\frac2\pi}\,
r^{\frac{2k_0-1}2} \cos\bigg(\frac{2k_0-1}{2}t\bigg) ,  \quad
\text{for all }r\in (0,1)\text{ and }t\in[0,\pi].
\end{equation}
\end{Lemma}
\begin{proof}
Let us set
\begin{equation}\label{eq:wtau}
w^\tau(z)=\frac{w(\tau z)}{\sqrt{H(\tau)}}.
\end{equation}
We notice that, for all $\tau\in (0,R)$, $w^\tau\in \mathcal H_1$
and $\int_{S_1^+}|w^{\tau}|^2ds=\int_0^\pi |w^\tau(\cos t,\sin
t)|^2\,dt=1$. Moreover, by scaling and \eqref{eq:9},
\begin{equation}\label{eq:8bis}
  \int_{B_{1}^+}\left( |\nabla w^\tau(z)|^2 -
\tau^{2} p(\tau z)|w^\tau(z) |^2\right)\,dz -\tau\int_0^1 q(\tau
x) |w^\tau(x,0) |^2\,dx={\mathcal N}(\tau)=O(1)
\end{equation}
as $\tau\to 0^+$, whereas from Lemmas \ref{l:hardy_poincare} and
\ref{l:hardy_boundary} it follows that
\begin{align}\label{eq:16}
  \mathcal N(\tau)&\geq \frac{1}{H(\tau)}
 \Big[1-
4\|p\|_{L^\infty(B_R^+)}\tau^2-\pi\|q\|_{L^\infty(\Gamma_n^R)}\tau\Big] \int_{B_{\tau}^+}|\nabla w|^2dz\\
\notag&= \Big[1-
4\|p\|_{L^\infty(B_R^+)}\tau^2-\pi\|q\|_{L^\infty(\Gamma_n^R)}\tau\Big]
\int_{B_{1}^+}|\nabla w^\tau|^2dz
\end{align}
for every $\tau\in(0,R_0)$, being $R_0$ as in \eqref{eq:r_0}. From
\eqref{eq:8bis}, \eqref{eq:16}, and Lemma \ref{l:hardy_poincare}
we deduce that
\begin{equation}\label{eq:19}
  \{w^\tau\}_{\tau\in(0,R_0)}\quad\text{is bounded in }H^1(B_1^+).
\end{equation}
Therefore, for any given sequence $\tau_n\to 0^+$, there exists a
subsequence $\tau_{n_k}\to0^+$ such that $w^{\tau _{n_k}}\weakly
\widetilde w$ weakly in $H^1(B_1^+)$ for some $\widetilde w\in
H^1(B_1^+)$.
 Due to compactness of trace embeddings, we
 have that $\widetilde w=0$ on $\Gamma_d^1$ and
 \begin{equation}\label{eq:10}
\int_{S_1^+}|\widetilde w|^2ds=1.
\end{equation}
 In particular $\widetilde w\not\equiv
0$. For every small  $\tau\in (0,R_0)$, $w^\tau$ satisfies
\begin{equation}\label{eqlam}
\begin{cases}
  -\Delta   w^\tau =\tau^2p(\tau z)w^\tau ,&\text{in }B_1^+,\\
\partial_\nu w^\tau  =\tau q(\tau x_1,0) w^\tau  ,&\text{on }\Gamma^1_n,\\
 w^\tau=0,&\text{on }\Gamma^1_d,
  \end{cases}
\end{equation}
in a weak sense, i.e.
\[
 \int_{B_1^+} \nabla w^\tau(z)\cdot\nabla \varphi(z)\,dz=
  \tau^2
\int_{B_1^+} p(\tau z)w^\tau(z)\varphi(z)\,dz+
 \tau\int_0^1 q(\tau x)w^\tau(x,0)\varphi(x,0)\,dx
\]
for all $\varphi \in H^1(B_1^+)$ s.t. $\varphi=0$ on
$S_1^+\cup\Gamma_d^1$. From weak convergence
 $w^{\tau _{n_k}}\weakly
\widetilde w$ in $H^1(B_1^+)$, we can pass to the limit in
\eqref{eqlam} along the sequence $\tau _{n_k}$ and obtain  that
$\widetilde w$ weakly solves
\begin{equation}\label{eq:extended_limit}
\begin{cases}
  -\Delta   \widetilde w=0,&\text{in }B_1^+,\\
\partial_\nu \widetilde w =0,&\text{on }\Gamma^1_n,\\
 \widetilde w=0,&\text{on }\Gamma^1_d.
  \end{cases}
\end{equation}
From \eqref{eq:19} it follows that $\{\tau q(\tau x)
w^\tau(x,0)\}_{\tau\in(0,R_0)}$ is bounded in
$H^{1/2}(\Gamma^1_n)$. Then, by elliptic regularity theory,
for every $0<r_1<r_2<1$ we have that $\{w^\tau\}_{\tau\in(0,R_0)}$
is bounded in $H^2(B^+_{r_2}\setminus \overline{B^+_{r_2}})$. From
compactness  of trace embeddings we have that, up to passing to a
further subsequence, $\frac{\partial
  w^{\tau_{n_k}}}{\partial\nu}\to \frac{\partial
  \widetilde w}{\partial\nu}$ in $L^2(S_r^+)$  for every $r\in
(0,1)$. Testing equation \eqref{eqlam} for $\tau=\tau_{n_k}$ with
$w^{\tau}$ on $B_r^+$ we obtain that
\begin{align*}
 \int_{B_r^+} |\nabla w^{\tau_{n_k}}(z)|^2\,dz&=
\int_{S_r^+}\frac{\partial
  w^{\tau_{n_k}}}{\partial\nu}w^{\tau_{n_k}}\,ds\\
&\qquad+  \tau_{n_k}^2 \int_{B_r^+} p(\tau_{n_k}
z)|w^{\tau_{n_k}}(z)|^2\,dz+
 \tau_{n_k}\int_0^r q(\tau_{n_k} x)|w^{\tau_{n_k}}(x,0)|^2\,dx\\
&\mathop{\rightarrow}\limits_{k\to+\infty}\int_{S_r^+}\frac{\partial
  \widetilde w}{\partial\nu}\widetilde w\,ds=\int_{B_r^+} |\nabla \widetilde w (z)|^2\,dz,
\end{align*}
thus proving that $\|w^{\tau _{n_k}}\|_{H^1(B_r^+)}\to
\|\widetilde w\|_{H^1(B_r^+)}$ for all $r\in(0,1)$, and hence
\begin{equation}\label{strongH1}
w^{\tau _{n_k}}\to \widetilde w\quad\text{in }H^1(B_r^+)
\end{equation}
for every $r\in (0,1)$. Furthermore, by compact Sobolev
embeddings, we also have that, up to extracting a further subsequence,
\begin{equation*}
  w^{\tau _{n_k}}\to
  \widetilde w
\quad\text{in }C^{0,\mu}_{\rm
loc}(\overline{B_r^+}\setminus\{0\}),
\end{equation*}
for every $r\in (0,1)$ and $\mu\in(0,1)$.

For any $r\in (0,1)$ and $k\in \N$, let us define the functions
\begin{align*}
&D_k(r)=\int_{B_r^+} |\nabla w^{\tau_{n_k}}|^2\,dz-
\tau_{n_k}^2\int_{B_r^+} p(\tau_{n_k} z)|w^{\tau_{n_k}}(z)|^2dz
 -\tau_{n_k}\int_0^r q(\tau_{n_k} x)|w^{\tau_{n_k}} (x,0)|^2\,dx,\\
&H_k(r)=\frac{1}{r}\int_{S_r^+}|w^{\tau_{n_k}}|^2 \, ds,
\end{align*}
and ${\mathcal
    N}_k(r):=\frac{D_k(r)}{H_k(r)}$.
Direct calculations yield that ${\mathcal
    N}_k(r)={\mathcal N}(\tau_{n_k}r)$ for all $r\in (0,1)$.
From (\ref{strongH1}) it follows that, for any fixed $r\in (0,1)$,
\begin{equation*}
 D_k(r)\to \widetilde D(r):=\int_{B_r^+}
|\nabla \widetilde w|^2\,dz\quad\text{and}\quad H_k(r)\to
\widetilde D(r):= \frac{1}{r}\int_{S_r^+}|\widetilde w|^2 \, ds.
\end{equation*}
From classical unique continuation principles for harmonic
functions it follows that $\widetilde D(r)>0$ and $\widetilde
H(r)>0$ for all $r\in(0,1)$ (indeed $\widetilde D(r)=0$ or
$\widetilde H(r)=0$ for some $r\in(0,1)$ would imply that
$\widetilde w \equiv 0$ in $B_r^+$ and, by unique continuation,
$\widetilde w \equiv 0$ in $B_1^+$, a contradiction). Hence, by
Lemma \ref{l:limitN},
\begin{equation}\label{Nw(r)}
\widetilde {\mathcal
    N}(r)=\frac{\widetilde D(r)}{\widetilde H(r)}=\lim_{k\to \infty}{\mathcal
    N}_k(r) =\lim_{k\to \infty} {\mathcal
    N}(\tau_{n_k}r)=\gamma
\end{equation}
for all $r\in (0,1)$. Therefore $\widetilde {\mathcal N}$ is
constant in $(0,1)$ and hence $\widetilde {\mathcal
  N}'(r)=0$ for any $r\in (0,1)$.  By \eqref{eq:extended_limit} and Lemma
\ref{mono} with $p\equiv 0$ and $q\equiv0$, we obtain
\begin{equation*}
\left(\int_{S_r^+}
  \left|\frac{\partial \widetilde w}{\partial \nu}\right|^2 ds\right) \cdot
    \left(\int_{S_r^+}
  \widetilde w^2\,ds\right)-\left(
\int_{S_r^+}
  \widetilde w\frac{\partial \widetilde w}{\partial \nu}\, ds\right)^{\!2} =0
\quad
  \text{for all } r\in (0,1),
\end{equation*}
which  implies  that $\widetilde w$ and $\frac{\partial \widetilde
w}{\partial \nu}$ are parallel as vectors in $L^2(S_r^+)$. Hence
there exists $\eta=\eta(r)$ such that $\frac{\partial \widetilde
w}{\partial \nu}(r\cos t,r\sin t)=\eta(r) \widetilde w(r\cos
t,r\sin t)$ for all $r\in(0,1)$ and $t\in[0,\pi]$. It follows that
\begin{equation} \label{separate}
\widetilde w(r\cos t,r\sin t) =\varphi(r) \psi(t), \quad
r\in(0,1), \ t\in [0,\pi],
\end{equation}
where $\varphi(r)=e^{\int_1^r \eta(s)ds}$ and $\psi(t)=\widetilde
w(\cos t,\sin t)$. From \eqref{eq:10} we have that $\int_0^\pi
\psi^2=1$. From \eqref{eq:extended_limit} and  (\ref{separate}) we
can conclude that
\[
\begin{cases}
\varphi''(r)\psi(t)+
\frac1r\varphi'(r)\psi(t)+\frac1{r^2}\varphi(r)\psi''(t)=0,&r\in(0,1),\quad
t\in[0,1],\\
\psi(\pi)=0,\\
\psi'(0)=0.
\end{cases}
\]
Taking $r$ fixed, we deduce that $\psi$ is necessarily an
eigenfunction of the eigenvalue problem \eqref{eq:2}. Then there
exists $k_0\in\N\setminus\{0\}$ such that
$\psi(t)=\pm\sqrt{\frac2\pi}\cos(\frac{2k_0-1}{2}t)$ and $\varphi(r)$
solves the equation
\[
\varphi''(r)+\frac{1}r\varphi'-\frac{(2k_0-1)^2}{4r^2}\varphi(r)=0.
\]
Hence $\varphi(r)$ is of the form
\[
\varphi(r)=c_1 r^{\frac{2k_0-1}2}+c_2 r^{-\frac{2k_0-1}2}
\]
for some $c_1,c_2\in\R$. Since the function
$r^{-\frac{2k_0-1}2}\psi(t)\notin H^1(B_1^+)$, we deduce that
necessarily $c_2=0$ and  $\varphi(r)=c_1 r^{\frac{2k_0-1}2}$.
Moreover, from $\varphi(1)=1$, we obtain that $c_1=1$ and then
\begin{equation} \label{expw}
\widetilde w(r\cos t,r\sin t)= \pm\sqrt{\frac2\pi}\,
r^{\frac{2k_0-1}2} \cos\bigg(\frac{2k_0-1}{2}t\bigg) ,  \quad
\text{for all }r\in (0,1)\text{ and }t\in[0,\pi].
\end{equation}
From \eqref{expw} it follows that
\[
  \widetilde  H(r)=\int_0^\pi \widetilde w^2(r\cos t,r\sin
  t)\,dt=r^{2k_0-1}.
\]
Hence, in view of   \eqref{H'2},
\[
\gamma=\widetilde {\mathcal N}(r)=\frac r2 \frac{\widetilde
H'(r)}{\widetilde
  H(r)}=\frac
r2(2k_0-1)\frac{r^{2k_0-2}}{r^{2k_0-1}}=\frac{2k_0-1}{2}.
\]
The proof of the lemma is thereby complete.
\end{proof}

\begin{Lemma} \label{l:limite}
Let $w\not\equiv 0$ satisfy \eqref{eq:5}, $H$ be defined in \eqref{H(r)}, and
$\gamma:=\lim_{r\rightarrow 0^+} {\mathcal
    N}(r)$ be as in Lemma \ref{l:limitN}. Then the limit
$\lim_{r\to0^+}r^{-2\gamma}H(r)$ exists and it is finite.
\end{Lemma}
\begin{proof}
In view of (\ref{1stest}) it is sufficient to prove that the limit
exists. By (\ref{H(r)}), (\ref{H'2}), and Lemma~\ref{l:limitN} we
have that
\[
\frac{d}{dr} \frac{H(r)}{r^{2\gamma}} =2r^{-2\gamma-1}
(D(r)-\gamma H(r))=2r^{-2\gamma-1} H(r) \int_0^r {\mathcal
N}'(\rho) d\rho,
\]
and then, by integration over $(r,R_0)$,
\begin{equation}\label{inte}
  \frac{H(R_0)}{R_0^{2\gamma}}-
  \frac{H(r)}{r^{2\gamma}}=2\int_r^{R_0}
  \frac{H(\rho)}{\rho^{2\gamma+1}} \left( \int_0^\rho \nu_1(t) dt \right) d\rho +2\int_r^{R_0}
  \frac{H(\rho)}{\rho^{2\gamma+1}} \left( \int_0^\rho \nu_2(t) dt \right) d\rho
\end{equation}
where $\nu_1$ and $\nu_2$ are as in (\ref{eq:nu1}) and
(\ref{eq:nu2}). Since, by Schwarz's inequality, $\nu_1\geq 0$, we
have that $\lim_{r\to 0^+} \int_r^{R_0}  \rho^{-2\gamma-1} H(\rho)
\left( \int_0^\rho
  \nu_1(t) dt \right) d\rho$
exists.  On the other hand, from Lemma \ref{l:limitN} $\mathcal N$
is bounded and hence from \eqref{eq:7} we deduce that $\nu_2$ is
bounded close to $0^+$. Hence, in view of
 \eqref{1stest}, the function $\rho\mapsto \rho^{-2\gamma-1} H(\rho)
 \left( \int_0^\rho \nu_2(t) dt \right)$ is bounded and hence
 integrable near $0$. We conclude that  both terms at the right hand side of
(\ref{inte}) admit a limit as $r\to 0^+$ thus completing the
proof.
\end{proof}

\noindent The following lemma provides some pointwise estimate for
solutions to \eqref{eq:5}.
\begin{Lemma}\label{l:stima}
 Let  $w\in H^1(B_R^+)$
be a nontrivial solution to  \eqref{eq:5}. Then there exist
$C_4,C_5>0$ and $\bar r\in(0,R_0)$ such that
\begin{enumerate}[\rm (i)]
\item $\sup_{S_{r}^+}|w|^2\leq
  \frac{C_4}{r}\int_{S_r^+}|w(z)|^2\,ds$ for every $0<r<\bar r$,
\item $|w(z)|\leq C_5|z|^\gamma$ for all $z\in B_{\bar r}^+$, with
$\gamma$ as in Lemma \ref{l:limitN}.
\end{enumerate}
\end{Lemma}
\begin{proof}
We first notice that (ii) follows directly from (i) and
\eqref{1stest}. In order to prove (i), we argue by contradiction
and assume that there exists a sequence $\tau_n\to 0^+$ such that
\begin{equation*}
  \sup_{t\in[0,\pi]}\Big|w\Big(\frac{\tau_n}2\cos t,\frac{\tau_n}2\sin t\Big)\Big|^2>n H \Big (\frac{\tau_n}2 \Big)
\end{equation*}
with $H$ as in \eqref{H(r)}, i.e., defining $w^\tau$ as in
\eqref{eq:wtau}
\begin{equation}\label{eq:18}
  \sup_{x\in S_{1/2}^+}|w^{\tau_n}(z)|^2>2n\int_{S_{1/2}^+}|w^{\tau_n}(z)|^2 ds.
\end{equation}
From Lemma \ref{l:blowup}, there exists a subsequence $\tau_{n_k}$
such that $w^{\tau_{n_k}}\to\widetilde w$ in  $C^{0}(S_{1/2}^+)$ with $\widetilde w$ being as in \eqref{eq:11},
hence passing to the limit in \eqref{eq:18} a contradiction
arises.
\end{proof}

\noindent To obtain a sharp asymptotics of $H(r)$ as $r\to 0^+$,
it remains to prove that $\lim_{r\to 0^+} r^{-2\gamma} H(r)$ is
strictly positive.

\begin{Lemma} \label{l:limitepositivo} Under the same assumptions as in
  Lemmas \ref{l:limite} and \ref{l:stima}, we have that
\[
\lim_{r\to0^+}r^{-2\gamma}H(r)>0.
\]
\end{Lemma}
\begin{proof}
From Lemma \ref{l:blowup} there exists $k_0\in\N$, $k_0\geq 1$
such that $\gamma=\frac{2k_0-1}{2}$. Let us expand $w$ as
\begin{equation}\label{eq:expansion}
w(r\cos t,r\sin t)=\sum_{k=1}^\infty\varphi_k(r)
\cos\left(\tfrac{2k-1}{2}t\right)
\end{equation}
where
\begin{equation}\label{eq:37}
  \varphi_k(r)=\frac2\pi\int_{0}^\pi w(r\cos t,r\sin t) \cos\left(\tfrac{2k-1}{2}t\right)\,dt.
  \end{equation}
The Parseval identity yields
\begin{equation}\label{eq:17bis}
H(r)=\frac\pi2 \sum_{k=1}^{\infty}\varphi_k^2(r),\quad\text{for
all }0<r\leq R.
\end{equation}
From \eqref{1stest} and \eqref{eq:17bis} it follows that, for all
$k\geq1$,
\begin{equation}\label{eq:23}
 \varphi_k(r)=O(r^{\gamma})\quad\text{as
 }r\to0^+.
\end{equation}
Let $\eta\in C^\infty_c(0,R)$. Testing \eqref{eq:5} with the function
$\eta(r)\cos\left(\frac{2k-1}2 t\right)$, by \eqref{eq:expansion} we
obtain
\begin{align} \label{eq:form-var}
& \frac \pi 2\int_0^R r\varphi_k'(r)\eta'(r)\, dr+\frac \pi 2\int_0^R \tfrac{(2k-1)^2}4 \frac 1r \varphi_k(r)\eta(r)\, dr
=\int_0^R q(r)w(r,0)\eta(r)\, dr \\
& \notag \qquad \qquad +\int_0^R r\eta(r)\left(\int_0^\pi p(r\cos t,r\sin t)w(r\cos t,r\sin t)\cos\left(\tfrac{2k-1}2\, t\right)\, dt\right)dr \, .
\end{align}
Integrating by parts in the first in integral on the left hand side of \eqref{eq:form-var} and exploiting the fact that $\eta\in C^\infty_c(0,R)$ is an arbitrary test function, we infer
\begin{equation*}
-\varphi_k''(r)-\frac{1}{r}\varphi_k'(r)+ \frac14(2k-1)^2
\frac{\varphi_k(r)}{r^2}=\zeta_k(r),\quad\text{in }(0,R),
\end{equation*}
where
\begin{equation}\label{eq:38}
  \zeta_k(r)=\frac2{\pi r}\,q(r) w(r,0)+\frac2\pi\int_0^\pi p(r\cos t,r\sin t) w(r\cos t,r\sin t)
\cos\left(\tfrac{2k-1}{2}t\right)\,dt.
\end{equation}
Then, by a  direct calculation, there exist $c_1^k,c_2^k\in\R$
such that
\begin{equation}\label{eq:42}
\varphi_k(r)=r^{\frac{2k-1}{2}}
\bigg(c_1^k+\int_r^R\frac{t^{\frac{1-2k}{2}+1}}{2k-1}
\zeta_k(t)\,dt\bigg)+ r^{\frac{1-2k}{2}}
\bigg(c_2^k+\int_r^R\frac{t^{\frac{2k-1}{2}+1}}{1-2k}
\zeta_k(t)\,dt\bigg).
\end{equation}
From Lemma \ref{l:stima} it follows that
\begin{equation}\label{eq:zeta}
  \zeta_{k_0}(r)=O\left(r^{\frac{2k_0-1}{2}-1}\right)\quad\text{as }r\to 0^+,
\end{equation}
and hence
 the functions
\[
t\mapsto t^{\frac{1-2k_0}{2}+1} \zeta_{k_0}(t)
\quad\text{and}\quad t\mapsto t^{\frac{2k_0-1}{2}+1}\zeta_{k_0}(t)
\]
belong to $L^1(0,R)$. Hence
\[
r^{\frac{2{k_0}-1}{2}}
\bigg(c_1^{k_0}+\int_r^R\frac{t^{\frac{1-2{k_0}}{2}+1}}{2{k_0}-1}
\zeta_{k_0}(t)\,dt\bigg) =o(r^{\frac{1-2{k_0}}{2}})\quad\text{as
}r\to0^+,
\]
and then, by \eqref{eq:23}, there must be
\begin{equation*}
c_2^{k_0}=\int_0^R\frac{t^{\frac{2{k_0}-1}{2}+1}}{2{k_0}-1}
\zeta_{k_0}(t)\,dt.
\end{equation*}
From  (\ref{eq:zeta}), we then deduce that
\begin{align}\label{eq:12}
r^{\frac{1-2k_0}{2}}
\bigg(c_2^{k_0}+\int_r^R\frac{t^{\frac{2{k_0}-1}{2}+1}}{1-2{k_0}}
\zeta_{k_0}(t)\,dt\bigg) &= r^{\frac{1-2k_0}{2}}
\int_0^r\frac{t^{\frac{2{k_0}-1}{2}+1}}{2{k_0}-1}
\zeta_{k_0}(t)\,dt=O(r^{k_0+\frac12})
\end{align}
as $r\to0^+$.  From (\ref{eq:42}) and (\ref{eq:12}), we obtain
that
\begin{equation}\label{eq:24}
\varphi_{k_0}(r)=r^{\frac{2{k_0}-1}{2}}
\bigg(c_1^{k_0}+\int_r^R\frac{t^{\frac{1-2{k_0}}{2}+1}}{2{k_0}-1}
\zeta_{k_0}(t)\,dt +O(r)\bigg)\quad\text{as }r\to0^+.
\end{equation}
Let us assume by contradiction that
$\lim_{r \to 0^+} r^{-2\gamma}H(r)=0$. Then
(\ref{eq:17bis}) would imply that
\begin{equation*}
\lim_{r\to0^+}r^{-\frac{2k_0-1}{2}}\varphi_{k_0}(r)=0,
\end{equation*}
and hence, in view of \eqref{eq:24}, we would have that
\[
c_1^{k_0}+\int_0^R\frac{t^{\frac{1-2{k_0}}{2}+1}}{2{k_0}-1}
\zeta_{k_0}(t)\,dt=0,
\]
which, together with (\ref{eq:zeta}), implies
\begin{align}\label{eq:13tris}
r^{\frac{2{k_0}-1}{2}}
\bigg(c_1^{k_0}+\int_r^R\frac{t^{\frac{1-2{k_0}}{2}+1}}{2{k_0}-1}
\zeta_{k_0}(t)\,dt \bigg) =r^{\frac{2{k_0}-1}{2}}
\int_0^r\frac{t^{\frac{1-2{k_0}}{2}+1}}{1-2{k_0}}
\zeta_{k_0}(t)\,dt =O(r^{\frac12+k_0})
\end{align}
as $r\to0^+$. From (\ref{eq:24}) and (\ref{eq:13tris}), we
conclude that $\varphi_{k_0}(r)=O(r^{\frac12+k_0})$ as $r\to0^+$,
namely,
\[
\sqrt{H(\tau)} \int_0^{\pi}w^\tau(\cos t,\sin
t)\cos\left(\tfrac{2k_0-1}{2}t\right)\,dt=
O(\tau^{\frac12+k_0})\quad\text{as }\tau\to0^+,
\]
where $w^\tau$ is defined in \eqref{eq:wtau}.
 From (\ref{2ndest}),
there exists $C>0$ such that $\sqrt{H(\tau)}\geq
C\tau^{\gamma+\frac12}$ for $\tau$ small, and therefore
\begin{equation}\label{eq:26}
\int_0^{\pi}w^\tau(\cos t,\sin
t)\cos\left(\tfrac{2k_0-1}{2}t\right)\,dt=
O(\tau^{\frac12})\quad\text{as }\tau\to0^+.
\end{equation}
From Lemma \ref{l:blowup},
 for every sequence $\tau_n\to0^+$, there exist a subsequence
$\{\tau_{n_k}\}_{k\in\N}$ such that
\begin{equation}\label{eq:25}
w^{\tau_{n_k}}(\cos t,\sin t)\to \pm\sqrt{\frac2\pi}\,
\cos\bigg(\frac{2k_0-1}{2}t\bigg) \quad\text{in } L^2(0,\pi).
\end{equation}
From (\ref{eq:26}) and (\ref{eq:25}), we infer that
\[
0=\lim_{k\to+\infty} \int_0^{\pi}w^{\tau_{n_k}} (\cos t,\sin
t)\cos\left(\tfrac{2k_0-1}{2}t\right)\,dt
=\pm\sqrt{\frac2\pi}\,\int_0^{\pi}\cos^2\left(\tfrac{2k_0-1}{2}t\right)\,dt=
\pm\sqrt{\frac\pi2},
\]
thus reaching a contradiction.
\end{proof}

\begin{proof}[Proof of Theorem \ref{t:asym}]
  Identity (\ref{eq:35}) follows from Lemma \ref{l:blowup}, thus there
  exists $k_0\in \N$, $k_0\geq 1$, such that
  $\gamma=\lim_{r\to 0^+}{\mathcal N}(r)= \frac{2k_0-1}{2}$.

Let $\{\tau_n\}_{n\in\N}\subset (0,+\infty)$ be such that
$\lim_{n\to+\infty}\tau_n=0$. Then, from Lemmas \ref{l:blowup} and
\ref{l:limitepositivo},  scaling and a diagonal argument,  there
exists a subsequence $\{\tau_{n_k}\}_{k\in\N}$ and $\beta\neq0$
such that
\begin{equation}\label{eq:15}
\frac{w(\tau_{n_k}z)}{\tau_{n_k}^\gamma}\to \beta
|z|^{\frac{2k_0-1}2} \cos\bigg(\tfrac{2k_0-1}{2}\mathop{\rm Arg}
z\bigg)
\end{equation}
 strongly in $H^1(B_r^+)$ for all $r>0$  and in  $C^{0,\mu}_{\rm
    loc}(\overline{\R^2_+}\setminus\{0\})$  for every $\mu\in (0,1)$. In particular
\begin{equation}\label{eq:13}
\tau_{n_k}^{-\gamma}w(\tau_{n_k}(\cos t,\sin
  t))\to \beta \cos\bigg(\frac{2k_0-1}{2}t\bigg)
\end{equation}
in  $C^{0,\mu}([0,\pi])$. To prove that the above converge
occurs as $\tau\to0^+$ and not only along subsequences, we are
going to show that $\beta$ depends neither on the sequence
$\{\tau_n\}_{n\in\N}$ nor on its subsequence
$\{\tau_{n_k}\}_{k\in\N}$.

   Defining $\varphi_{k_0}$ and $\zeta_{k_0}$
as in (\ref{eq:37}) and \eqref{eq:38}, from (\ref{eq:13}) it
follows that
\begin{align}\label{eq:43}
  \frac{\varphi_{k_0}(\tau_{n_k})}{\tau_{n_k}^{\gamma}} &=
  \frac2\pi\int_{0}^\pi\frac{w(\tau_{n_k}\cos t,\tau_{n_k}\sin
    t)}{\tau_{n_k}^{\gamma}} \cos\left(\tfrac{2k_0-1}{2}t\right)\,dt\\
&\notag\to
\frac2\pi\beta\int_{0}^\pi\cos^2\left(\tfrac{2k_0-1}{2}t\right)\,dt
=\beta
\end{align}
as $k\to+\infty$. On the other hand, from \eqref{eq:42},
\eqref{eq:12} , and \eqref{eq:24} we know that that
\begin{align}\label{eq:14}
  \varphi_{k_0}(\tau)&=\tau^{\frac{2k_0-1}{2}}
\bigg(c_1^{k_0}+\int_\tau^R\frac{t^{\frac{1-2k_0}{2}+1}}{2k_0-1}
\zeta_{k_0}(t)\,dt\bigg)+\tau^{\frac{1-2k_0}{2}}
\int_0^\tau\frac{t^{\frac{2{k_0}-1}{2}+1}}{2{k_0}-1}
\zeta_{k_0}(t)\,dt\\
\notag&=\tau^{\frac{2{k_0}-1}{2}}
\bigg(c_1^{k_0}+\int_\tau^R\frac{t^{\frac{1-2{k_0}}{2}+1}}{2{k_0}-1}
\zeta_{k_0}(t)\,dt +O(\tau)\bigg)\quad\text{as }\tau\to0^+.
\end{align}
Choosing $\tau=R$ in the first line of \eqref{eq:14}, we obtain
\[
c_1^{k_0} =R^{-\frac{2k_0-1}{2}}\varphi_{k_0}(R)-R^{1-2k_0}
\int_0^R\frac{t^{\frac{2{k_0}-1}{2}+1}}{2{k_0}-1}
\zeta_{k_0}(t)\,dt .
\]
Hence, from the second line of \eqref{eq:14}, we obtain that
\[
\tau^{-\gamma}\varphi_{k_0}(\tau)\to
R^{-\frac{2k_0-1}{2}}\varphi_{k_0}(R)-R^{1-2k_0}
\int_0^R\frac{t^{\frac{2{k_0}-1}{2}+1}}{2{k_0}-1}
\zeta_{k_0}(t)\,dt+\int_0^R\frac{t^{\frac{1-2{k_0}}{2}+1}}{2{k_0}-1}
\zeta_{k_0}(t)\,dt,
\]
as $\tau\to0^+$. Then, from (\ref{eq:43}) we deduce that
\begin{equation}\label{eq:31}
  \beta=R^{-\frac{2k_0-1}{2}}\varphi_{k_0}(R)-R^{1-2k_0}
\int_0^R\frac{t^{\frac{2{k_0}-1}{2}+1}}{2{k_0}-1}
\zeta_{k_0}(t)\,dt+\int_0^R\frac{t^{\frac{1-2{k_0}}{2}+1}}{2{k_0}-1}
\zeta_{k_0}(t)\,dt.
\end{equation}
In particular  $\beta$ depends neither on the sequence
$\{\tau_n\}_{n\in\N}$ nor on its subsequence
$\{\tau_{n_k}\}_{k\in\N}$, thus implying that the convergence in
(\ref{eq:15}) actually holds as $\tau\to 0^+$ and proving the
theorem. We observe that \eqref{eq:21} follows by replacing
\eqref{eq:37} and \eqref{eq:38} into  \eqref{eq:31}.
\end{proof}

\section{Some regularity estimates}\label{sec:some-regul}

In this section, we prove some regularity and
  approximation results, which will be used to estimate the H\"older
  norm of the difference between a solution $u$ to \eqref{eq:22} and
  its asymptotic profile $\beta F_{k_0}$.
\begin{Proposition}\label{eq:reg-estim}
Let $f\in L^\infty (B_4^+)$, $g\in L^\infty (\G_n^4)$ and let $v\in
H^1(B_4)\cap L^\infty(B_4^+)$ solve
\begin{equation}\label{eq:v-solv-app-1}
\begin{cases}
  -\Delta   v=f,&\text{in }B_4^+,\\
\partial_\nu v =g,&\text{on }\Gamma^4_n,\\
v=0,&\text{on }\Gamma^4_d.
  \end{cases}
\end{equation}
Then, for every $\e>0$, there exists a constant $C>0$
  (independent of $v$, $f$, and $g$) such that \be
\|v\|_{C^{1/2-\e} (\ov{B_2^+})}\leq C \left( \|f\|_{L^\infty(B_4^+)}+
  \|g\|_{L^\infty(\G^4_n)}+ \|v\|_{L^\infty(B_4^+)} \right) .  \ee
\end{Proposition}
\begin{proof}
In the sequel we denote as $C>0$ a positive constant
  independent of $v$,  $f$, and $g$ which may vary from line to line.
  We consider a $C^2$  domain $\O'$ such that $B_3^+\subset \O'\subset B_4^+$ and $\G^3_n\cup \G^3_d\subset \de \O'.$
 We define the functions (obtained uniquely  by  minimization arguments) $v_1 \in H^1 (\O')$ satisfying
\begin{equation}\label{eq:v1-solv-app}
\begin{cases}
  -\Delta   v_1=  f,&\text{in } \O' ,\\
\partial_\nu v_1  =0,&\text{  on $\G^{{3}}_n$ } ,\\
v_1  =0,& \text{ on $\partial\O'\setminus \G^3_n$ },
  \end{cases}
\end{equation}
and   $\ti v_2\in H^{1/2}(\R)$ satisfying
$$
\begin{cases}
  (-\D)^{\frac{1}{2}} \ti  v_2=g,&\text{in }  (0,4) ,\\
\ti v_2  =0,&  \text{on $\R \setminus (0,4)$ }.
%
  \end{cases}
$$
Therefore  by   (fractional) elliptic regularity theory (see
e.g. \cite[Proposition 1.1]{ROS2014}),  we deduce that
\be\label{eq:tiv2-C-12}
\|\ti v_2 \|_{C^{1/2}({\R})} \leq C  \|  g\|_{L^\infty(\G_n^4)}.
\ee
Consider the Poisson kernel $P(x_1,x_2)= \frac 1\pi x_2  |x|^{-2}$ with respect to the half-space $\R^2_+$, see \cite[Section 2.4]{cs}. We define
$$
v_2(x_1,x_2)=(P(\cdot, x_2)\star \ti v_2)( x_1)=\frac 1\pi x_2\int_{\R}\frac{\ti v_2(t)}{x_2^2+(x_1-t)^2 }\, dt= \frac 1\pi \int_{\R }\frac{\ti v_2( x_1-r x_2)}{1+r^2 }\, dr
$$
where with the symbol $\star$ we denoted the convolution product with respect to the first variable.
One can check that $v_2\in H^1_{\rm loc}(\overline{\R^2_+})$
 (see for example \cite[Subsection 2.1]{bcd}) and
\begin{equation}\label{eq:v2-solv-app}
\begin{cases}
  -\Delta   v_2=0,&\text{in }  \R^2_+ ,\\
\partial_\nu v_2  =  g,& \text{on $\G^4_n$ },\\
v_2  =0,&  \text{on $\R \setminus (0,4)$}.
%
  \end{cases}
\end{equation}
It is easy to see that
$$
\|v_2\|_{L^\infty(\R^2_+)}\leq C\|\ti v_2\|_{L^\infty(\R)}.
$$
Moreover  by \eqref{eq:tiv2-C-12},  for $x,y\in \ov{\R^2_+}$ we get
\begin{align*}
|v_2(x)- v_2(y)|&\leq  C \|
                  g\|_{L^\infty(\G_n^4)}|x-y|^{1/2}\int_{\R
                  }\frac{\max(1, |r|^{1/2})}{1+r^2 }\, dr\\
&\leq C  \|  g\|_{L^\infty(\G_n^4)}|x-y|^{1/2}.
\end{align*}
It follows that
\be \label{eq:est-v2}
\|v_2\|_{C^{1/2}(\ov{\R^2_+})} \leq C \|  g\|_{L^\infty(\G_n^4)} .
\ee
By \cite[Theorem 1]{Savare}  and continuous embeddings of
  Besov spaces into H\"older spaces, we get
$$
\|v_1\|_{C^{1/2-\e}(\overline{\Omega'})}^2 \leq C \|v_1\|_{H^1(\O')}
\left( \| f\|_{L^\infty(B_4^+)} + \|v_1\|_{H^1(\O')} \right) .
$$
Multiplying \eqref{eq:v1-solv-app} by $v_1$,  integrating by parts and using Young's inequality, we get
$$
C \| v_1\|_{L^2(\O')}^2 \leq \|\n v_1\|_{L^2(\O')}^2\leq \| v_1\|_{L^2(\O')}  \| f\|_{L^2(B_4^+)}  \leq    \e \| v_1\|_{L^2(\O')}^2+ C_\e   \| f\|_{L^\infty(B_4^+)}^2,
$$
where in the first estimate we have used the Poincar\'e inequality
for functions vanishing on a portion of the boundary.
 We then conclude that
\be \label{eq:est-v1}
\|v_1\|_{C^{1/2-\e}(\ov{\Omega'})}  \leq   C  \| f\|_{L^\infty(B_4^+)}.
\ee
Now, thanks to \eqref{eq:v-solv-app-1}, \eqref{eq:v1-solv-app} and \eqref{eq:v2-solv-app}, the
function $V:=v-(v_1+v_2)\in H^1(\Omega')$ solves the
equation
\begin{equation}\label{eq:v-solv-app-2}
\begin{cases}
  -\Delta   V=0,&\text{in }\Omega',\\
\partial_\nu V =0,&\text{on }\Gamma^3_n,\\
V=0,&\text{on }\Gamma^3_d.
  \end{cases}
\end{equation}
By elliptic regularity theory, we have that
\be
\label{eq:est-for-V-0}
\|V\|_{C^{2}(\ov{ B_{5/2}^+\setminus B_1^+} )} \le C
\|V\|_{H^1(B_r^+)}
\ee
where $r$ is a fixed radius satisfying $\frac 52<r<3$ and $C>0$ is independent of $V$. Let $\eta$ a radial cutoff function compactly supported in $B_3$ satisfying $\eta\equiv 1$ in $B_r$; testing \eqref{eq:v-solv-app-2} with $\eta V$, we infer that $\|V\|_{H^1(B_r^+)}\le C\|V\|_{L^2(\Omega')}$ for some constant $C>0$ independent of $V$. Hence by \eqref{eq:est-for-V-0} we obtain
\be
\label{eq:est-for-V}
\|V\|_{C^{2}(\ov{ B_{5/2}^+\setminus B_1^+} )} \leq C
\|V\|_{L^\infty(\Omega')}.
\ee
 Let
$\ti \eta\in C^\infty_c(B_{5/2})$ be a radial function, with
$\ti \eta\equiv 1$ on $B_2$. Then the function
$\ti V:=\ti \eta V\in H^1(\R^2_+)$ solves
\begin{equation*} 
\begin{cases}
  -\Delta  \ti  V= - V\D \ti \eta- 2\n V \cdot \n \ti\eta,&\text{in }\R^2_+,\\
\partial_\nu \ti V(x_1,0) =0,& x_1\in (0,+\infty),\\
\ti V(x_1,0)=0,& x_1\in (-\infty,0) .
  \end{cases}
\end{equation*}
Then by  \cite[Theorem 1]{Savare}, the arguments above, \eqref{eq:est-for-V}, \eqref{eq:est-v2} and  \eqref{eq:est-v1}, we deduce that
\begin{align*}
\| v-(v_1+v_2) \|_{C^{1/2-\e}(\ov{B_2^+})} &\leq \|\ti V
\|_{C^{1/2-\e}(\ov{\R^2_+})} \leq C
                                             \|V\|_{L^\infty(\Omega')} \\
&\leq C
\left( \|f\|_{L^\infty(B_4^+)}+ \|g\|_{L^\infty(\G^4_n)}+
  \|v\|_{L^\infty(B_4^+)} \right) .
\end{align*}
This, combined again with \eqref{eq:est-v2} and \eqref{eq:est-v1} completes the proof.
\end{proof}

Recalling \eqref{eq:def-F-k}, for every $k\in \N$ with $k\geq1$,  we consider the finite dimensional linear subspace of $L^2(B_r^+)$, given by
$$
\cS_k:=\left\{\sum_{j=1}^k a_j F_j\,:\, (a_1,\dots, a_k)\in \R^k \right\}.
$$
%
 For every $r>0$, $k\geq1$, and $u\in L^2(B_r^+)$, we  let
\[
F_{k,r}^u:= \textrm{Argmin}_{F\in \cS_k}\int_{B_r^+}(u(x)-F(x))^2\, dx
\]
be the $L^2(B_r^+)$-projection of $u$ on $\cS_k$, so that
\[
  \min_{F\in \cS_k}\int_{B_r^+}(u(x)-F(x))^2\, dx=   \int_{B_r^+}(u(x)-F_{k,r}^u(x))^2\, dx
\]
and
\be\label{eq:equilib}
\int_{B_r^+}(u(x)-F_{k,r}^u(x))F(x)\, dx=0,\quad\textrm{for all $F\in\cS_k$}.
\ee
Next, we estimate the $L^\infty$ norm of the difference between a solution of a mixed boundary value problem on $B_1^+$ and
its projection on $\cS_k$.

\begin{Proposition}\label{prop:sharp-estim-u-flat}
Let   $u\in H^1(B_1^+)\cap L^\infty(\R^2_+)$ solve
\begin{equation}\label{eq:v-solv-app}
\begin{cases}
  -\Delta   u=f,&\text{in }B_1^+,\\
\partial_\nu u =g,&\text{on }\Gamma^1_n,\\
u=0,&\text{on }\Gamma^1_d,
  \end{cases}
\end{equation}
where, for some $k\in \N\setminus\{0\}$ and $\ov C>0$,
\begin{align*}
&|f(x)| \leq \ov C  | x|^{\max(\g_k-\frac{3}{2},0)},  \quad\textrm{for every $x\in B_1^+$},\\
&|g(x_1)|\leq \ov C | x_1|^{\max({\g_k-\frac{1}{2}},0)}, \quad\textrm{for
  every $ x_1\in (0, 1)$},
\end{align*}
and  $\gamma_k=\frac{2k-1}{2}$.
Then, for every $\a\in (0,1/2)$, we have  that
\be
\sup_{r>0}r^{-\g_k-\a}\|u-F_{k,r}^u  \|_{L^\infty(B_r^+)} <\infty.
\ee
\end{Proposition}
\begin{proof}
In the sequel, $C>0$ stands for a positive constant, only depending on $\a,\ov C$ and $k$, which may vary from line to line.
 Assume by contradiction that, there exists $\a\in( 0,1/2)$ such that
\be
\sup_{r>0}r^{-\g_k-\a}\|u-F_{k,r}^u  \|_{L^\infty(B_r^+)} =\infty.
\ee
 We consider the nonincreasing  function
\be
\Theta(r):=\sup_{\ov r> r} \ov r^{-\g_k-\a}  \| u- F_{k,\ov r}^u\|_{L^\infty(B_{\ov r}^+)}.
\ee
It is clear from our assumption that
$$
\Theta(r)  \nearrow +\infty  \qquad\textrm{ as $r\to 0$}.
$$
Then there exists a sequence $r_n \to 0 $ such that
$$
   r^{-\g_k-\a}_n\| u- F_{k,r_n}^u\|_{L^\infty(B_{r_n}^+)}\geq \frac{\Theta(r_n)}{2}.
$$
%
 %
 %
We define
\begin{align*}
v_n(x)&:=   r^{-\g_k-\a}_n \frac{ u(r_n x)-  F_{k,r_n}^u(r_n x) }{\Theta(r_n)} ,
\end{align*}
so that
\be \label{eq:contradict-2}
    \| v_n\|_{L^\infty(B_1^+)}    \geq \frac{1}{2}.
\ee
Moreover,  by a change of variable in \eqref{eq:equilib}, we get
\be\label{eq:contradict-3}
\int_{B_1^+}v_n(x)F(x)\, dx=0 \qquad\textrm{ for every $F\in \cS_k$}.
\ee
\noindent
\textbf{Claim:} For $R=2^m$ and $r>0$, we have
\be\label{eq:FRF}
 \frac{1}{r^{\g_k+\a}\Theta(r)}\| F^u_{k, r R}-F^u_{k, r}\|_{L^\infty(B_{  rR}^+)}\leq C R^{\g_k+\a}.
\ee
Indeed, by definition, for every $\ov r> r>0,  $ we have
$$
   \| u- F_{k,\ov r}^u\|_{L^\infty(B_{\ov r}^+)}\leq \ov r^{\g_k+\a} \Theta(r)
$$
and thus,  using the monotonicity of $\Theta$,  for every $x\in B_r^+$ we get
\begin{align}\label{eq:F2r-Fr}
|  F_{k,2 r}^u(x) - F_{k,  r}^u(x)| &\leq     \| u- F_{k,2
                                      r}^u\|_{L^\infty(B_{ 2r}^+)}+ \|
                                      u- F_{k, r}^u\|_{L^\infty(B_{
                                      r}^+)}\\
&\notag \leq  2^{1+\g_k+\a}r^{\g_k+\a} \Theta(r)\leq C r^{\g_k+\a} \Theta(r).
\end{align}
Letting $F_{k,  r}^u=\sum_{j=1}^k a_j(r)F_j$ and $\gamma_j=\frac{2j-1}{2}$, by taking the $L^2(B_r^+)$-norm in \eqref{eq:F2r-Fr}, we get
\be
|  a_j(2r)-a_j(r)| r^{\g_j}\leq C r^{\g_k+\a} \Theta(r) \qquad\textrm{ for every $r>0$.}
\ee
 Then
\begin{align*}
 \frac{1}{r^{\g_k+\a}\Theta(r)}\| F^u_{k, r  2^m}-F^u_{k, r}\|_{L^\infty(B_{ r 2^m}^+ )}& \leq   \frac{1}{r^{\g_k+\a}\Theta(r)} \sum_{j=1}^{k} |  a_j(r 2^m)-a_j(r)| (r2^m)^{\g_j} \\
& \leq \frac{1}{r^{\g_k+\a}\Theta(r)} \sum_{j=1}^{k} \sum_{i=1}^{m}  | a_j(2^i r) - a_j(2^{i-1} r)|  (r 2^m)^{\g_j} \\
&\leq  \frac{C}{r^{\g_k+\a}\Theta(r)}  \sum_{j=1}^{k} \sum_{i=1}^{m} 2^{\g_j m}   2^{(\g_k-\g_j+\a)(i-1)}r^{\g_k +\a} \Theta(2^{i-1}r)\\
&\leq  {C}   \sum_{j=1}^{k} \sum_{i=1}^{m} 2^{\g_j m}   2^{(\g_k-\g_j+\a)(i-1)}\leq  {C}   \sum_{j=1}^{k}   2^{\g_j m}   2^{(\g_k-\g_j+\a) m}  \\
&\leq  C 2^{m(\g_k+\a)}  .
\end{align*}
This proves the \textbf{claim}.\\

From  the definition of $\Theta$ and \eqref{eq:FRF}, for $R= 2^m\geq 1$,  we have
\begin{align*}
\sup_{x\in B_R^+}|v_n(x)|&= \frac{1}{r^{\g_k+\a}_n\Theta(r_n)}\| u-F^u_{k, r_n}\|_{L^\infty(B_{r_n R}^+)}\\
&\leq   \frac{1}{r^{\g_k+\a}_n\Theta(r_n)}\| u-F^u_{k, r_n
  R}\|_{L^\infty(B_{r_n R}^+)}+
 \frac{1}{r^{\g_k+\a}_n\Theta(r_n)}\| F^u_{k, r_n R}-F^u_{k, r_n}\|_{L^\infty(B_{ r_n R}^+)}\\
&\leq  \frac{ 1}{ r^{\g_k+\a}_n \Theta(r_n) }  (r_nR)^{\g_k+\a}  \Theta( r_n) +  CR^{\g_k+\a}\\
&\leq  CR^{\g_k+\a}.
\end{align*}
Consequently, letting $R\geq 1$ and $m_0\in \N$ be the smallest
integer such that $2^{m_0}\geq R$, we obtain that
\begin{align} \label{eq:bound}
\sup_{x\in B_R^+}|v_n(x)|&\leq
\sup_{x\in B_{2^{m_0}}^+}|v_n(x)|\leq
C2^{m_0(\gamma_k+\alpha)}\leq C(2R)^{\gamma_k+\alpha}\\
\notag &\leq
CR^{\g_k+\a},
\end{align}
 with $C$ being a positive constant independent of $R$. Thanks to \eqref{eq:eq-for-F-k} and
\eqref{eq:v-solv-app}, it is plain that
$$
\begin{cases}
  -\Delta   v_n =\frac{r_n^{2-\g_k-\a}}{\Theta(r_n)} f (r_n \cdot ),&\text{in }B_{{1}/{r_n}}^+,\\[5pt]
\partial_\nu v_n  = \frac{r_n^{1-\g_k-\a}}{\Theta(r_n)} g (r_n \cdot ),&\text{on }\Gamma^{{1}/{r_n}}_n,\\[5pt]
v_n =0,&\text{on }\Gamma^{{1}/{r_n}}_d.
  \end{cases}
$$
By assumption, we have that $\frac{r_n^{2-\g_k-\a}}{\Theta(r_n)} f (r_n x)$
and $ \frac{r_n^{1-\g_k-\a}}{\Theta(r_n)} g (r_n x_1) $ are bounded in
$L^\infty(B_M^+)$ and $L^\infty(\G^M_n)$ respectively, for every $M>0$. Hence, by Proposition
\ref{eq:reg-estim} and \eqref{eq:bound}, we have that $v_n$ is bounded in
$C^{\delta}(\overline{B_M^+})$ for every $M>0$ and
$\delta\in(0,1/2)$. Furthermore, it is easy to verify that $v_n$ is
bounded in $H^1(B_M^+)$ for every $M>0$.  Then, for every $M>0$ and
$\delta\in(0,1/2)$, $v_n$ converges in $C^{\delta}(\overline{B_M^+})$
(and weakly in $H^1(B_M^+)$) to some
$v\in C^{\delta}_{\rm loc}(\ov{\R^2_+})\cap H^1_{\rm loc}(\R^2_+)$
satisfying
$$
\begin{cases}
  -\Delta   v=0,&\text{in } \R^2_+,\\
\partial_\nu v   = 0,&\text{on }\Gamma^{\infty}_n,\\
v=0,&\text{on }\Gamma^{\infty}_d,
  \end{cases}
$$
and by \eqref{eq:bound}, for every $R>1$,
$$
\|v\|_{L^\infty(B_R^+)}\leq  C R^{\g_k+\a}.
$$
By Lemma \ref{lem:liouville} (below), we deduce that necessarily
\[
 v\in \cS_k.
\]
 This clearly yields a contradiction when passing to the limit  in \eqref{eq:contradict-2} and   \eqref{eq:contradict-3}.
\end{proof}
The following Liouville type result was used in the   proof of Proposition \ref{prop:sharp-estim-u-flat}.
\begin{Lemma}[Liouville theorem] \label{lem:liouville}
Let
$v\in C(\ov{\R^2_+})\cap H^1_{\rm loc}(\R^2_+)$ satisfy
$$
\begin{cases}
  -\Delta   v=0,&\text{in } \R^2_+,\\
\partial_\nu v   = 0,&\text{on }\Gamma^{\infty}_n,\\
v=0,&\text{on }\Gamma^{\infty}_d,
  \end{cases}
$$
and, for some $\a\in (0,1/2)$ and $C>0$,
\begin{equation}\label{eq:28}
\|v\|_{L^\infty(B_R^+)}\leq C\,R^{\g_k+\a}\quad\textrm{for every $R>1$},
\end{equation}
where $\gamma_k=\frac{2k-1}{2}$, $k\in\N\setminus\{0\}$.
Then
\be
 v\in \cS_k.
\ee
\end{Lemma}

\begin{proof}
Arguing as in the proof of Lemma \ref{l:limitepositivo}, we expand $v$
in Fourier series with respect to the orthonormal basis of
$L^2(0,\pi)$ given in \eqref{eq:17} as
\begin{equation*}
v(r\cos t,r\sin t)=\sum_{j=1}^\infty\varphi_j(r)
\cos\left(\tfrac{2j-1}{2}t\right)
\end{equation*}
where $\varphi_j(r)=\frac2\pi\int_{0}^\pi v(r\cos t,r\sin t) \cos\left(\tfrac{2j-1}{2}t\right)\,dt$.
From assumption \eqref{eq:28} and the Parseval identity we have that
\begin{equation*}
\frac\pi2 \sum_{j=1}^{\infty}\varphi_j^2(r)=\int_0^\pi v^2(r\cos
t,r\sin t)\,dt\leq\pi C^2 r^{2(\gamma_k+\alpha)},\quad\text{for
all }r>1.
\end{equation*}
It follows that
\begin{equation}\label{eq:29}
|\varphi_j(r)|\leq {\rm const\,}r^{\gamma_k+\alpha}\quad\text{for all
}j\geq1\text{ and }r>1,
\end{equation}
for some ${\rm const\,}>0$ independent of $j$ and $r$.

From the equation satisfied by $v$ it follows that the functions
$\varphi_j$ satisfy
\begin{equation*}
-\varphi_j''(r)-\frac{1}{r}\varphi_j'(r)+ \frac14(2j-1)^2
\frac{\varphi_j(r)}{r^2}=0,\quad\text{in }(0,+\infty),
\end{equation*}
and then, for all $j\geq1$, there exist $c_1^j,c_2^j\in\R$
such that
\begin{equation*}
\varphi_j(r)=c_1^j r^{\frac{2j-1}{2}}+
c_2^j r^{\frac{1-2j}{2}}\quad\text{for all }r>0.
\end{equation*}
The fact $v$ is continous and $v(0)=0$ implies that
$\varphi_j(r)=o(1)$ as $r\to 0^+$. As a consequence we have that
$c_2^j=0$ for all $j\geq1$. On the other hand \eqref{eq:29} implies
that $c_1^j=0$ for  all $j>k$. Therefore we conclude that
\begin{equation*}
v(r\cos t,r\sin t)=\sum_{j=1}^k c_1^j r^{\frac{2j-1}{2}}
\cos\left(\tfrac{2j-1}{2}t\right)=\sum_{j=1}^k c_1^j F_j(r\cos t,r\sin t),
\end{equation*}
i.e. $v\in \cS_k$.
\end{proof}

\section{Asymptotics for $u$}\label{sec:asymptotics-u}

We are now in position to prove Theorem \ref{t:main-u}.
\begin{proof}[Proof of Theorem \ref{t:main-u}]
Let $w=u\circ\varphi^{-1}$, with $\varphi:\overline{\mathcal U_R}\to \overline{B_R^+}$ being the
  conformal map constructed in Section \ref{sec:auxiliary-problem}.
  Let
$\g=\frac{2k_0-1}{2}$, with $k_0$ being as in Theorem
  \ref{t:asym}.
  We define  (recalling \eqref{eq:wtau})
\[
\widetilde w^\tau(z):=\tau^{-\gamma}w(\tau z)=\tau^{-\gamma}\sqrt{H( \t )} w^\t(z) .
\]
From Theorem \ref{t:asym} we have that there exists $\beta\neq0$ such
that $\widetilde w^\tau\to \beta F_{k_0}$
 in $H^1(B_r^+)$ for all $r>0$  and in  $C^{0,\mu}_{\rm
    loc}(\overline{\R^2_+}\setminus\{0\})$ for every $\mu\in(0,1)$.  \\

\noindent
\textbf{Claim 1:}   We have
\be \label{eq:wy-to-beta}
w(y)= \b F_{k_0}(y)+o(|y|^\g) \qquad\textrm{ as $|y|\to 0$  and  $y\in {B_R^+}$}.
\ee
If this does not hold true then there exists a sequence of points $y_m\in (B_R^+\cup \G_n^R)\setminus \{0\}$ and $C>0$ such that $y_m\to 0$ and
$$
 |y_m|^{-\gamma}  |  w(y_m)-\b F_{k_0}(y_m)| =   |   \ti{w}^{\t_m}(  z_m)-\b F_{k_0}(z_m)|\geq C>0,
$$
where  $\t_m=|y_m| $ and $ z_m= \frac{y_m}{|y_m|}$. If $m$ is large enough, we get a contradiction with \eqref{eq:shar-asym-w-tau}.  This proves \eqref{eq:wy-to-beta} as   claimed.

Let $\varrho\in(0,1/2)$ and let $p$ and $q$ be the functions introduced in \eqref{eq:5}. By \eqref{eq:wy-to-beta},
by the fact that $p\in L^\infty(B_R^+)$ and $q\in C^1([0,R))$, and by Proposition \ref{prop:sharp-estim-u-flat} applied to $w$, we have that, for every $r\in (0,R)$,
\be \label{eq:w-est-Fwk0}
|w(x)-F_{k_0,r}^w(x)| \leq C r^{\g+\varrho}, \quad\textrm{for every $x\in B_r^+$,}
\ee
for some  positive constant $C>0$ independent of $r$, which could vary from line to line in the sequel.

From \eqref{eq:wy-to-beta} and \eqref{eq:w-est-Fwk0} we deduce that
\be \label{eq:F-k0-Fw-k0}
\sup_{x\in B_r^+} r^{-\gamma}|\b F_{k_0}(x)-F_{k_0,r}^w(x)|\to 0, \quad\textrm{as $r \to 0^+$.}
\ee

\noindent
\textbf{Claim 2:} We have
\be \label{eq:claim-to-prove-F}
|\b F_{k_0}(x)-F_{k_0,r}^w(x)| \leq  C r^{\g+\varrho}, \quad\textrm{for every $x\in B_r^+$}.
\ee
Once this claim is proved, then according to \eqref{eq:w-est-Fwk0}, we can easily deduce that for any $r\in (0,R)$
$$
|w(x)- \b F_{k_0}(x) |\leq |w(x)-  F_{k_0,r}^w(x) |+|F_{k_0,r}^w- \b
F_{k_0}(x) |
 \leq C r^{\g+\varrho},\quad\textrm{for every $x\in B_r^+$}.
$$
In particular,
$$
|w(x)- \b F_{k_0}(x) | \leq C |x|^{\g+\varrho},\quad\textrm{for every $x\in B_{R}^+$}
$$
which finishes   the proof of Theorem \ref{t:main-u}.

\noindent
Let us now prove \textbf{Claim 2}.
Writing $F_{k_0,r}^w(x)=\sum_{j=1}^{k_0}a_j(r) F_j(x) $,  by
\eqref{eq:F-k0-Fw-k0} we have that
\be \label{eq:ak0-beta0}
| \b-a_{k_0}(r) |\to 0,\quad\text{as }r\to0^+ .
\ee
Moreover by taking the $L^2(B^+_r)$-norms   in \eqref{eq:F-k0-Fw-k0}, we find that
$$
(a_{k_0}(r)-\b)^2 r^{2\g+2}+\sum_{j=1}^{k_0-1} a_j^2(r)r^{2\g_j+2}
\leq C r^{2\g+ 2}, \quad\textrm{for every $R>r >0$},
$$
with  $\gamma_j=\frac{2j-1}{2}$.
This   yields, for $j=1,\dots,  k_0-1$,
\be  \label{eq:aj-0}
  |a_j(r)|\leq C r^{\g-\g_j}\to 0  \qquad\textrm{ as $r\to 0$.}
\ee
  From    \eqref{eq:w-est-Fwk0}, we get, for every $x\in B_r^+$ and $R>r>0$,
$$
\bigg|  w(x)- \sum_{j=1}^{k_0}a_j(r) F_j(x)\bigg|\leq   C r^{\varrho+\g}.
$$
Hence, for every $x\in B_{r/2}^+$, we have that
$$
\bigg|  \sum_{j=1}^{k_0}(a_j(r)-a_j(2^{-1}r)) F_j(x)   \bigg|\leq  |  F_{k_0,r}^w(x) -   w (x)  |+   |  F_{k_0, 2^{-1}r}^w(x)  -  w(x)  |\leq C r^{\varrho+\g}.
$$
Taking  the $L^2(B^+_{r/2})$-norms in the previous inequality,  we find that, for every $r\in (0,R)$
$$
\sum_{j=1}^{k_0}| a_{j}(r)  - a_{j}(2^{-1}r)| r^{\g_j}   \leq    C r^{\g+\varrho}.
$$
This implies that
$$
| a_{j}(r)  - a_{j}(2^{-1}r)| \leq C
r^{\varrho+\g-\g_j}\qquad\textrm{for all $1\leq j\leq k_0$ and $r\in(0,R)$.}
$$
From this, \eqref{eq:ak0-beta0} and \eqref{eq:aj-0},   we obtain
\begin{align*}
 |\b - a_{k_0}(r)| r^{-\varrho}+ \sum_{j=1}^{k_0-1} |a_j(r)| r^{-\varrho-\g+\g_j}&\leq \sum_{j=1}^{k_0} \sum_{i=0}^\infty |a_{j}(r 2^{-i-1} )- a_{j}(r 2^{-i})|r^{-\varrho-\g+\g_j} \\
 & \leq C  \sum_{i=0}^\infty  2^{-i\varrho}.
 \end{align*}
This   implies that, for every $x\in B_r^+$,
$$
|\b F_{k_0}(x)-F_{k_0,r}^w(x)|\leq   |\b - a_{k_0}(r)| r^{\g}+ \sum_{j=1}^{k_0-1} |a_j(r)| r^{ \g_j} \leq C r^{\g+\varrho}.
$$
That is  \eqref{eq:claim-to-prove-F} as claimed.
\end{proof}

\begin{remark}
\begin{enumerate}
\item[(i)]  Since $\vp $ is conformal,  we have that $\ti F:= F_{k_0}\circ\vp$ satisfies
 $\ti F\in H^1 (\cU_R)$ and solves the homogeneous  equation
\begin{align}
\begin{cases}
\D \ti  F =0,& \textrm{in $\cU_R$},\\
 \ti  F =0,& \textrm{on   $\G_d\cap\de \cU_R$}\\
\de_{\nu }\ti F=0,& \textrm{on   $\G_n\cap\de \cU_R$}.
\end{cases}
\end{align}
\item[(ii)] Let $\Upsilon: U^+:=\cB\cap U\to B_\rho^+$ define a $C^2$ parametrization  (e.g. given by a system of  \textit{Fermi coordinates}),  for some open neighborhood  $U$ of 0,  with $\Upsilon(0)=0$, $D \Upsilon(0)=Id$,   $\Upsilon(\G_n\cap  U  )\subset \G_n^\rho$ and $\Upsilon(\G_d\cap  U  )\subset \G_d^\rho$. By Theorem \ref{t:main-u},  for every  $\varrho\in (0,1/2)$, there exist $C,\rho_0>0$ such that
\be \label{eq:asympt-u-gen-cor}
 |  u( \Upsilon^{-1}(y))-\b \a^{\frac{2k_0-1}{2}}  F_{k_0}(y)|\leq C |y|^{ \frac{2k_0-1}{2}+ \varrho }, \quad \textrm{for every $y\in   B_{\rho_0}^+$,}
\ee
with $\alpha>0$ as in \eqref{eq:20}. Indeed, to see this, we first observe that \eqref{eq:asympt-u-gen-cor} is equivalent to
\be \label{eq:asympt-u-gen-cor-inv}
 |  u(x)-\b  F_{k_0}(\a \Upsilon(x))|\leq c |x|^{\frac{2k_0-1}{2}+\varrho}, \quad \textrm{for every $x\in   \Upsilon^{-1}(B_{\rho_0}^+)$,}
\ee
for some constant $c>0$.
We then further note that
$$
|D F_{k_0}(x)|\leq c |x|^{\frac{2k_0-1}{2}-1}
$$
and  thus
\begin{align*}
|F_{k_0}(\a \Upsilon(x))-F_{k_0}(\vp(x))&|\leq  c
                                          |x|^{\frac{2k_0-1}{2}-1}|\a
                                          \Upsilon(x) -\vp(x)|\\
&\leq  c |x|^{\frac{2k_0-1}{2}-1} |x|^2\\
&\leq  c |x|^{\frac{2k_0-1}{2}+1},
\end{align*}
in a neighborhood of $0$, where $c>0$ is  a positive constant
  independent of $x$ possibly varying
  from line to line.
This, together with \eqref{eq:asympt-u-sharp} and the triangular inequality, gives \eqref{eq:asympt-u-gen-cor-inv}.
\end{enumerate}
\end{remark}

\begin{proof}[Proof of Corollary \ref{c:upper_bound}]
  From Theorem \ref{t:main-u} and \eqref{eq:asympt-u-gen-cor} it follows that, if $u\in H^1(\Omega)$
  is a non-trivial solution to \eqref{eq:22}, then there exist $k_0\in
  \N\setminus\{0\}$ and $\beta\in\R\setminus\{0\}$ such that, for every $t\in[0,\pi)$,
\begin{equation}\label{eq:30}
\lim_{r\to 0}r^{-\frac{2k_0-1}{2}} u(r\cos t,r\sin
  t)=\beta\alpha^{\frac{2k_0-1}{2}}\cos\left(\tfrac{2k_0-1}{2}t\right).
\end{equation}
Therefore, if $u\geq0$, we have that necessarily $k_0=1$ so that
statement (i) follows. Moreover, \eqref{eq:asympt-u-gen-cor} implies
that
\[
u(r\cos t,r\sin t)\geq \beta\alpha^{1/2}r^{1/2}\cos\left(\tfrac{t}{2}\right)
-Cr^{1/2+\varrho},
\]
which easily provides statement (ii).
\end{proof}

\begin{proof}[Proof of Corollary \ref{c:unique_continuation}]
  Let us assume by contradiction that $u\not\equiv 0$. Then, Theorem
  \ref{t:main-u} and \eqref{eq:asympt-u-gen-cor} imply that
  \eqref{eq:30} holds  for every $t\in[0,\pi)$ and
for some $k_0\in \N\setminus\{0\}$ and $\beta\in\R\setminus\{0\}$.
Taking $n>\frac{2k_0-1}2$, \eqref{eq:30} contradicts the assumption
that $u(x)=O(|x|^n)$ as $|x|\to 0$.
\end{proof}


\section{An example}\label{sec:examples}
In this section we show that the presence of a logarithmic term in
the asymptotic expansion cannot be excluded without assuming
enough regularity of the boundary.

Let us define in the Gauss plane the set
\begin{equation*}
A:=\C\setminus\{x_1\in \R\subset \C:x_1\leq 0\}
\end{equation*}
and the holomorphic function $\eta:A\to\C$ defined as follows:
\begin{align*}
&\eta(z):=\log r+i\theta \quad \text{for any $z=re^{i\theta}\in
A$,
 $r>0$,
$\theta\in\left(-\pi,\pi\right)$}.
\end{align*}
Let us consider the holomorphic function
$$
v(z):=e^{2\eta(-iz)}\eta(-iz) \qquad \text{for
  any } z\in \C\setminus \{ix_2: x_2 \le 0\}
$$
and the set
\begin{equation} \label{Zeta}
\mathcal Z:=\{z\in \C\setminus \{ix_2: x_2 \le 0\}:\Im(v(z))=0\} .
\end{equation}
If $z=re^{i\theta}$ with $r>0$, $\theta\in \left(-\frac
\pi2,\frac{3\pi}2\right)\setminus \{-\frac \pi 4,0,\frac \pi
4,\frac{\pi}2,\frac{3\pi}4,\pi,\frac{5\pi}4\}$, then $z\in \mathcal Z$
if and only
\begin{equation} \label{rho_1} r=\rho(\theta):=
  \exp\left[-\bigg(\theta-\frac{\pi}2\bigg)\cot(2\theta) \right] .
\end{equation}
For some fixed $\sigma\in \left(0,\frac\pi 2\right)$, we define
the curve $\Gamma_+\subset \mathcal Z$ parametrized by
\begin{equation} \label{gamma+}
\Gamma_+:
\begin{cases}
x_1(\theta)=\rho(\theta)\cos\theta \\
x_2(\theta)=\rho(\theta)\sin\theta
\end{cases}
\qquad \theta\in (-\sigma,0) \, .
\end{equation}
If we choose $\sigma>0$ sufficiently small then $\Gamma_+$ is the
graph of a function $h_+$ defined in a open right
neighborhood $U_+$ of $0$. Moreover $h_+$ is a Lipschitz
function in $U_+$, $h_+\in C^2(U_+)$ and
\begin{equation} \label{derivatives}
\lim_{x_1\to 0^+} \frac{h_+(x_1)}{x_1}=0 \, , \quad
\lim_{x_1\to 0^+} h_+'(x_1)=0 \, .
\end{equation}
Then we define the harmonic function
\begin{equation} \label{eq:def-u-ex}
u(x_1,x_2):=-\Im(v(z)) \qquad \text{ for any } z=x_1+ix_2\in
\C\setminus \{iy: y \le 0\} \, .
\end{equation}
In polar coordinates the function $u$ reads
\begin{equation} \label{polar}
u(r,\theta)=r^2 \left[(\log r)
\sin(2\theta)+\left(\theta-\frac{\pi}2\right) \cos(2\theta)\right]
.
\end{equation}
From (\ref{Zeta}--\ref{rho_1}) and \eqref{polar} we deduce that
$u$ vanishes on $\Gamma_+$.

The next step is to find a curve $\Gamma_-$ on which
$\frac{\partial u}{\partial\nu}=0$ where $\nu=(\nu_1,\nu_2)$ is
the unit normal to $\Gamma_-$ satisfying $\nu_2\le 0$. We observe
that
\begin{equation*}
u(x_1,x_2)=x_1
x_2\log(x_1^2+x_2^2)+\left[\arctan\left(\frac{x_2}{x_1}\right)+\frac
\pi 2\right](x_1^2-x_2^2) \quad \text{for any } x_1<0 , \, x_2\in
\R \, .
\end{equation*}
From direct computation we obtain
\begin{align*}
& \frac{\partial u}{\partial
x_1}(x_1,x_2)=x_2\log(x_1^2+x_2^2)+x_2
+2\left[\arctan\left(\frac{x_2}{x_1}\right)+\frac \pi 2\right]x_1 \, , \\[10pt]
& \frac{\partial u}{\partial
x_2}(x_1,x_2)=x_1\log(x_1^2+x_2^2)+x_1
-2\left[\arctan\left(\frac{x_2}{x_1}\right)+\frac \pi 2\right]x_2
\, .
\end{align*}
We now define
\begin{equation*}
H_1(x_1,x_2)=\frac{2\left[\arctan\left(\frac{x_2}{x_1}\right)+\frac
\pi 2\right]x_1}{\log(x_1^2+x_2^2)} \quad \text{and} \quad
H_2(x_1,x_2)=\frac{2\left[\arctan\left(\frac{x_2}{x_1}\right)+\frac
\pi 2\right]x_2}{\log(x_1^2+x_2^2)}
\end{equation*}
on the set $B_1\cap \Pi_-$ where $\Pi_-:=\{(x_1,x_2)\in\R^2:x_1<0\}$.
One can easily check that $H_1,H_2$ admit continuous extensions
defined on $B_1\cap \overline \Pi_-$ which we still denote by $H_1$
and $H_2$ respectively. We also observe that $H_1,H_2\in
C^1(B_1\cap \overline \Pi_-)$. Therefore $H_1,H_2$ may be extended
also on the right of the $x_2$-axis up to restrict them to a disk
of smaller radius. For example one may define
\begin{equation*}
H_1(x_1,x_2):=3H_1(-x_1,x_2)-2H_1(-2x_1,x_2) \quad \text{and}
\quad H_2(x_1,x_2):=3H_2(-x_1,x_2)-2H_2(-2x_1,x_2)
\end{equation*}
for any $(x_1,x_2)\in B_{1/2}\cap \Pi_+$ where we put
$\Pi_+:=\{(x_1,x_2)\in\R^2:x_1>0\}$. One may check that the new
functions $H_1,H_2$ belong to $C^1(B_{1/2})$.

We can now define the functions $V_1,V_2:B_{1/2}\to \R$ by
\begin{align*}
& V_1(x_1,x_2):=
\begin{cases}
x_2+\frac{x_2}{\log(x_1^2+x_2^2)}+H_1(x_1,x_2) & \quad \text{if } (x_1,x_2)\neq (0,0) \\
0           & \quad \text{if } (x_1,x_2)=(0,0) \, ,
\end{cases} \\
& V_2(x_1,x_2):=
\begin{cases}
x_1+\frac{x_1}{\log(x_1^2+x_2^2)}-H_2(x_1,x_2) & \quad \text{if } (x_1,x_2)\neq (0,0) \\
0           & \quad \text{if } (x_1,x_2)=(0,0) \, .
\end{cases}
\end{align*}
One may verify that $V_1,V_2\in C^1(B_{1/2})$. Moreover we have
\begin{align*}
& \frac{\partial V_1}{\partial x_1}(0,0)=0 \, , \quad
\frac{\partial V_1}{\partial x_2}(0,0)=1 \, , \quad \frac{\partial
V_2}{\partial x_1}(0,0)=1 \, , \quad \frac{\partial V_2}{\partial
x_2}(0,0)=0 \, .
\end{align*}
Then we consider the dynamical system
\begin{equation} \label{eq:dyn-sys}
\begin{cases}
x_1'(t)=V_1(x_1(t),x_2(t)) \\[5pt]
x_2'(t)=V_2(x_1(t),x_2(t)) \, .
\end{cases}
\end{equation}
After linearization at $(0,0)$, by \cite[Theorem IX.6.2]{Hartman}
 we deduce that the stable and unstable manifolds
corresponding to the stationary point $(0,0)$ of
\eqref{eq:dyn-sys}, are respectively tangent to the eigenvectors
$(1,-1)$ and $(1,1)$ of the matrix $DV(0,0)$ where $V$ is the
vector field $(V_1,V_2)$.

We define the curve $\Gamma_-$ as the stable manifold of \eqref{eq:dyn-sys} at $(0,0)$
intersected with $B_\varepsilon \cap \Pi_-$ where
$\varepsilon\in (0,\frac 12)$ can be chosen sufficiently small in such
a way that $\Gamma_-$ becomes the graph of a function $h_-$
defined in a open left neighborhood $U_-$ of $0$. Combining the
definitions of $h_+$ and $h_-$ we can introduce a
function $h:U_+\cup U_-\cup \{0\}\to \R$ such that
$h\equiv h_+$ on $U_+$, $h\equiv h_-$ on $U_-$
and $h(0)=0$.

Then we introduce a positive number $R$ sufficiently small and a
domain $\Omega\subseteq B_R$ such that $\Omega=\{(x_1,x_2)\in
B_R:x_2>h(x_1)\}$. One can easily check that the function $u$
defined in \eqref{eq:def-u-ex} belongs to $H^1(\Omega)$. From the
above construction, we deduce that $u=0$ on $\Gamma_+\cap \partial\Omega$
and $\frac{\partial u}{\partial \nu}=0$ on $\Gamma_-\cap \partial\Omega$.
We observe that $\partial\Omega$ admits a corner at $0$ of
amplitude $\frac{3\pi}4$.

The presence of a logarithmic term in $u$ can be explained since the
$C^{2,\delta}$-regularity assumption is not satisfied from the right, i.e. $h_{|U_+\cup \{0\}}\not\in C^{2,\delta}(U_+\cup \{0\})$ for any $\delta\in (0,1)$.
To see this, it is sufficient to study the
behavior of $h(x_1)-x_1 h'(x_1)$ in a right neighborhood
of zero.

By \eqref{gamma+} we know that $\theta\in
\big(-\frac{\pi}2,0\big)$ and hence, if $x_1$ belongs to a
sufficiently small right neighborhood of $0$, by \eqref{rho_1} we
have
\begin{equation} \label{identity} \frac{1}2 \log
  \big(x_1^2+(h_+(x_1))^2\big)
  \tan\left[2\arctan\left(\frac{h_+(x_1)}{x_1}\right)\right]
  +\arctan\left(\frac{h_+(x_1)}{x_1}\right)-\frac{\pi}2=0 .
\end{equation}
By \eqref{derivatives} and \eqref{identity} we have that, as
$x_1\to 0^+$,
\begin{align} \label{tangent} \tan & \left[2\arctan\left(\frac{h_+(x_1)}{x_1}\right)\right]
  =-\frac{2\arctan\big(\frac{h_+(x_1)}{x_1}\big)-\pi}{\log
    \big(x_1^2+(h_+(x_1))^2\big)}
   =\frac{\pi}2 \frac{1}{\log
    x_1}+o\left(\frac{1}{\log x_1}\right)   .
\end{align}
Differentiating both sides of \eqref{identity} and multiplying by
$x_1^2+(h_+(x_1))^2$ we obtain the identity
\begin{multline} \label{eq:derivative}
  \big(x_1+h_+(x_1)h_+'(x_1)\big) \, \tan\left[2\arctan\left(\frac{h_+(x_1)}{x_1}\right)\right] \\
  +\left\{ 1+\frac{\log \big(x_1^2+(h_+(x_1))^2\big)}
    {\cos^2\left[2\arctan\left(\frac{h_+(x_1)}{x_1}\right)\right]} \right\}
  \big(x_1h_+'(x_1)-h_+(x_1) \big)=0
\end{multline}
and hence \eqref{derivatives} and \eqref{tangent} yield
\begin{equation} \label{stima} x_1h_+'(x_1)-h_+(x_1)\sim
  -\frac{\pi}4 \,
  \frac{x_1}{\log^2 x_1} \qquad \text{as } x_1\to 0^+ .
\end{equation}
This shows that $h_+ \not \in C^2(U_+\cup\{0\})$ (and a fortiori cannot be extended to be of class $C^{2,\delta}$).

We observe that the reason of the appearance of a logarithmic term is not due to the presence of a corner at
$0$; indeed we are going to construct a domain with $C^1$-boundary for
which the same phenomenon occurs. In order to do this, it is
sufficient to take the domain $\Omega$ and the function $u$
defined above and to apply a suitable deformation in order to
remove the angle. We recall that $\Omega$ exhibits a corner at $0$ whose amplitude is $\frac{3\pi}4$.

For this reason, we define $F:\C\setminus\{ix_2:x_2\le 0\}\to \C$ by
\begin{equation*}
F(z):=r^{\frac 43} \, e^{i\frac 43 \theta} \quad \text{for any }
z=re^{i\theta} \, , \ r>0\, , \theta\in
\left(-\frac{\pi}2,\frac{3\pi}2\right) .
\end{equation*}
We observe that, up to shrink $R$ if necessary, the map
$F:\Omega\to F(\Omega)$ is invertible so that we may define
$\widetilde\Omega:=F(\Omega)$ and $\widetilde u:\widetilde
\Omega\to \R$, $\widetilde u(y_1,y_2):=u(F^{-1}(y_1,y_2))$ for any
$(y_1,y_2)\in \widetilde\Omega$.

We also define the curves $\widetilde\Gamma_+:=F(\Gamma_+)$ and
$\widetilde\Gamma_-:=F(\Gamma_-)$. Up to shrink $R$ if necessary,
we may assume that $\widetilde\Gamma_+$ and $\widetilde\Gamma_-$
are respectively the graphs of two functions $\widetilde h_+$
and $\widetilde h_-$.

It is immediate to verify that $\widetilde u=0$ on
$\widetilde\Gamma_+$. We also prove that $\frac{\partial\widetilde
u}{\partial \nu}=0$ on $\widetilde\Gamma_-$. To avoid confusion
with the notion of normal unit vectors to $\Gamma_-$ and
$\widetilde\Gamma_-$ we denote them respectively with
$\nu_{\Gamma_-}$ and $\nu_{\widetilde\Gamma_-}$.
Since $\widetilde u$ is still harmonic, $\frac{\partial u}{\partial \nu_{\Gamma_-}}=0$ on $\Gamma_-$ and $F$ is a conformal mapping, for any $\widetilde \varphi\in C^\infty_c(\widetilde\Omega\cup \widetilde \Gamma_-)$, we have
\begin{align*}
& \int_{\widetilde\Gamma_-} \frac{\partial \widetilde u}{\partial\nu_{\widetilde \Gamma_-}}\widetilde \varphi\, ds
=\int_{\widetilde\Omega} \nabla\widetilde u(y)\nabla \widetilde \varphi(y)\, dy
=\int_{\widetilde\Omega} [\nabla u(F^{-1}(y))(DF(F^{-1}(y)))^{-1}] \nabla \widetilde \varphi(y)\, dy \\[7pt]
& =\int_{\Omega} \big[\nabla u(x)(DF(x))^{-1}\big] \nabla \widetilde \varphi(F(x))\, |\mathop{\rm det}(DF(x))| \, dx \\[7pt]
& =\int_{\Omega} \big[\nabla u(x)(DF(x))^{-1}\big] \big[\nabla\varphi(x)(DF(x))^{-1}\big]\, |\mathop{\rm det}(DF(x))| \, dx \\[7pt]
& =\int_{\Omega} \nabla u(x)\nabla\varphi(x) \, dx=\int_{\Gamma_-} \frac{\partial u}{\partial \nu_{\Gamma_-}}\varphi \, ds=0
\end{align*}
where we put $\varphi(x)=\widetilde \varphi(F(x))$. This proves that $\frac{\partial\widetilde u}{\partial \nu_{\widetilde \Gamma_-}}=0$ on $\widetilde \Gamma_-$.

Finally we prove for $\widetilde h_+$ an estimate similar to
\eqref{stima}.

From the definition of $F$ it follows that $\widetilde\Gamma_+$
admits a representation in polar coordinates of the type
\begin{equation}
r=\widetilde\rho(\theta):=\exp\left[-\left(\theta-\frac{2\pi}3\right)\cot\left(\frac{3\theta}2\right)\right]
\, .
\end{equation}
Proceeding exactly as for \eqref{identity}-\eqref{tangent} one can
prove that
\begin{equation} \label{identity-2} \frac{1}2 \log
  \big(x_1^2+(\widetilde h_+(x_1))^2\big)
  \tan\left[\frac 32\arctan\left(\frac{\widetilde h_+(x_1)}{x_1}\right)\right]
  +\arctan\left(\frac{\widetilde h_+(x_1)}{x_1}\right)-\frac{2\pi}3=0 \, .
\end{equation}
As we did for $h_+$, also for the function $\widetilde h_+$ one can prove that
\begin{equation} \label{eq:h-beha}
\lim_{x_1\to 0} \frac{\widetilde h_+(x_1)}{x_1}=0 \, , \quad
\lim_{x_1\to 0^+} \widetilde h_+'(x_1)=0 \, .
\end{equation}
By \eqref{eq:h-beha} we have
\begin{align} \label{tangent-2}
\tan  \left[\frac 32\arctan\left(\frac{\widetilde h_+(x_1)}{x_1}\right)\right]
  &=-\frac{2\arctan\big(\frac{\widetilde h_+(x_1)}{x_1}\big)-\frac{4\pi}3}{\log
    \big(x_1^2+(\widetilde h_+(x_1))^2\big)}\\
&\notag   =\frac{2\pi}3 \frac{1}{\log
    x_1}+o\left(\frac{1}{\log x_1}\right) \quad \text{as } x_1\to 0^+ \,   .
\end{align}
Differentiating both sides of \eqref{identity-2} and multiplying
by $x_1^2+(\widetilde h_+(x_1))^2$ we obtain the identity
\begin{multline} \label{eq:derivative-2}
  \big(x_1+\widetilde h_+(x_1)\widetilde h_+'(x_1)\big) \,
  \tan\left[\frac 32\arctan\left(\frac{\widetilde h_+(x_1)}{x_1}\right)\right] \\
  +\left\{ 1+\frac{3\log \big(x_1^2+(\widetilde h_+(x_1))^2\big)}
    {4\cos^2\left[\frac 32\arctan\left(\frac{\widetilde h_+(x_1)}{x_1}\right)\right]} \right\}
  \big(x_1\widetilde h_+'(x_1)-\widetilde h_+(x_1) \big)=0
\end{multline}
By \eqref{eq:h-beha}, \eqref{tangent-2} and \eqref{eq:derivative-2}, we obtain
\begin{equation} \label{stima-2} x_1\widetilde h_+'(x_1)-\widetilde h_+(x_1)\sim
  -\frac{4\pi}9 \,
  \frac{x_1}{\log^2 x_1} \qquad \text{as } x_1\to 0^+ .
\end{equation}
The above arguments show that $\partial\widetilde\Omega$ is of class
$C^1$ but not of class $C^{1,\delta}$ (and a fortiori not of class $C^{2,\delta}$).

\bigskip

{\bf Acknowledgments} M.M. Fall is supported by the Alexander von
Humboldt foundation.  V. Felli is partially supported by the PRIN2015
grant ``Variational methods, with applications to problems in
mathematical physics and geometry''.  A. Ferrero is partially
supported by the PRIN2012 grant ``Equazioni alle derivate parziali di
tipo ellittico e parabolico: aspetti geometrici, disu\-gua\-glian\-ze
collegate, e applicazioni'' and by the Progetto di Ateneo 2016 of the
University of Piemonte Orientale ``Metodi analitici, numerici e di
simulazione per lo studio di equazioni differenziali a derivate
parziali e applicazioni''.  A. Ferrero and V. Felli are partially
supported by the INDAM-GNAMPA 2017 grant ``Stabilit\`{a} e analisi
spettrale per problemi alle derivate parziali''.

\end{document}